\documentclass[twoside,11pt]{article}

\usepackage{jmlr2e}
\usepackage{multirow}
\usepackage{amsmath}

\usepackage{amsthm}
\usepackage{mathtools}
\usepackage{amsxtra}
\usepackage{hyperref}
\usepackage{paralist}
\usepackage{threeparttable}
\usepackage{cleveref}
\usepackage{cite}
\usepackage{color}
\usepackage{float}
\usepackage{subcaption}
\usepackage[justification=centering]{caption}
\renewcommand{\thesection}{\arabic{section}}

\newtheorem{Theorem}{Theorem}[section]
\newtheorem{Lemma}[Theorem]{Lemma}
\newtheorem{Remark}[Theorem]{Remark}
\newtheorem{assump}{Assumption}
\newenvironment{myassump}[2][]
{\renewcommand\theassump{#2}\begin{assump}[#1]}
	{\end{assump}}

\def \btheta{\boldsymbol{\theta}}

\jmlrheading{0}{0000}{0-00}{0/00}{0/00}{00000}{Anit Kumar Sahu, Dusan Jakovetic and Soummya Kar}

\ShortHeadings{Communication Optimality Trade-offs For Distributed Estimation}{Sahu, Jakovetic and Kar}
\firstpageno{1}

\begin{document}

\title{Communication Optimality Trade-offs For Distributed Estimation}

\author{\name Anit Kumar Sahu \email anits@andrew.cmu.edu \\
       \addr Department of Electrical and Computer Engineering\\
       Carnegie Mellon University\\
       Pittsburgh, PA 15213, USA
       \AND
       \name Dusan Jakovetic \email djakovet@uns.ac.rs\\
       \addr Department of Mathematics and Informatics\\
       Faculty of Sciences, University of Novi Sad\\
       21000 Novi Sad, Serbia
       \AND
       \name Soummya Kar \email soummyak@andrew.cmu.edu \\
       \addr Department of Electrical and Computer Engineering\\
       Carnegie Mellon University\\
       Pittsburgh, PA 15213, USA}

\editor{}

\maketitle

\begin{abstract}
This paper proposes $\mathbf{C}$ommunication efficient $\mathbf{RE}$cursive $\mathbf{D}$istributed estimati$\mathbf{O}$n algorithm, $\mathcal{CREDO}$, for networked multi-worker setups without a central master node. $\mathcal{CREDO}$ is designed for scenarios in which the worker nodes aim to collaboratively estimate a vector parameter of interest using distributed online time-series data at the individual worker nodes. The individual worker nodes iteratively update their estimate of the parameter by assimilating latest locally sensed information and estimates from neighboring worker nodes exchanged over a (possibly sparse) time-varying communication graph. The underlying inter-worker communication protocol is adaptive, making communications increasingly (probabilistically) sparse as time progresses. Under minimal conditions on the inter-worker information exchange network and the sensing models, almost sure convergence of the estimate sequences at the worker nodes to the true parameter is established.
Further, the paper characterizes the performance of $\mathcal{CREDO}$ in terms of asymptotic covariance of the estimate sequences and specifically establishes the achievability of optimal asymptotic covariance. The analysis reveals an interesting interplay between the algorithm's communication cost~$\mathcal{C}_{t}$ (over $t$ time-steps) and the asymptotic covariance. Most notably, it is shown that $\mathcal{CREDO}$ may be designed to achieve a $\Theta\left(\mathcal{C}_{t}^{-2+\zeta}\right)$ decay of the mean square error~($\zeta>0$, arbitrarily small) at each worker node, which significantly improves over the existing $ \Theta\left(\mathcal{C}_{t}^{-1}\right)$ rates.
Simulation examples on both synthetic and real data sets demonstrate $\mathcal{CREDO}$'s communication efficiency.
\end{abstract}

\begin{keywords}
Distributed Estimation, Statistical Inference, Stochastic Approximation, Networks, Optimization
\end{keywords}

\section{Introduction}
\label{sec:intro}
Distributed data processing techniques have been increasingly employed to solve problems pertaining to optimization and statistical inference. With massive computing resources that are available at scale, and ever growing sizes of data sets, it becomes highly desirable, if not necessary, to distribute the task among multiple machines or multiple cores. The benefits of splitting the task into smaller subtasks are multi-pronged, namely, it makes the problem at hand, scalable, parallelized and fast. In the context of distributed stochastic optimization, several methods~(see, for example \citet{zhang2013divide,zhang2013information,heinze2016dual,ma2015adding,recht2011hogwild}) have been proposed which exhibit impressive performance in platforms such as Mapreduce and Spark. The aforementioned methods, though highly scalable, are designed for master-worker or similar types of architectures. That is, they require the presence of a master node, i.e., a central coordinator which is tasked with splitting the dataset by data points~(batches) or by features among worker nodes and enabling the read/write operations of the iterates of the worker nodes so as to ensure information fusion across the worker nodes. However, with several emerging applications, master-worker type architectures may not be feasible or desirable due to physical constraints. Specifically, we are interested in systems and applications where the entire data is not available at a central/master node, is sensed in a streaming fashion and is intrinsically distributed across the worker nodes. Such scenarios arise, e.g., in systems which involve Internet of Things~(IoT). For example, a smart campus with sensors of various kinds, a smart building or monitoring a large scale industrial plant. Therein, a network of large number of heterogeneous entities~(usually, geographically spread) connected in a arbitrary network structure individually perform sensing for data arriving in a streaming fashion. The sensing devices have limited communication capabilities owing to on board power constraints and harsh environments.
A typical IoT framework is characterized by a heterogeneous network of entities without a central coordinator, where entities have localized knowledge and can exchange information among each other through an arbitrary pre-specified communication graph. Furthermore, the data samples arrive in a streaming fashion. The ad-hoc nature of the IoT framework necessitates the information exchange in a crafted manner, rather than just a single or few rounds of communication at the end as in \citet{zhang2013divide,zhang2013information,heinze2016dual,ma2015adding}.

Distributed algorithms for statistical inference and optimization in the aforementioned frameworks are characterized by central coordinator-less recursive procedures, where each entity in the network maintains its own estimate or optimizer for the problem at hand. Also, due to heterogeneity of the entities and lack of global model information, the information exchange is limited to the iterates and not the raw data. This additionally enhances privacy as far as individual worker raw data is concerned. In particular, the diffusion and consensus+innovations schemes have been extensively used for various distributed inference problems in the aforementioned frameworks, which include distributed parameter estimation, distributed detection and multi-task learning, to name a few~(see, for example, \citet{kar2008distributed,Sayed-LMS,cattivelli2010diffusion,kar2011convergence,bajovic2015distributed,sahu2015distributed,jakoveticxaviermoura-11,chen2014multitask}). Other variants of distributed recursive algorithms of such kinds have generated a lot of interest of late~(see, for example \citet{nedic2014nonasymptotic,ram2010distributed,braca2008enforcing,Nedic-parameter,Ram-Nedich-Siam,Nedic-opt,jadbabaie2012non}).\\
An entity or node in an IoT setup is usually equipped with on board communication and computation units. However, finite battery power calls for frugal communication protocols as the power used in communication tends to beat the power required for on board computation. Thus, \emph{communication efficiency} is highly relevant and sought for in such scenarios. As far as distributed parameter estimation is concerned, the previously studied distributed algorithms mentioned above, have the mean square error (MSE) of estimation decay as $\Theta(\mathcal{C}_{t}^{-1})$, in terms of the communication cost $\mathcal{C}_{t}$, measured as the total number of communications between neighboring worker nodes over $t$ (discrete) time-steps and assuming each worker node obtains an independent measurement sample at each time. In this paper, we present a distributed recursive algorithm, {\bf C}ommunication Efficient {\bf RE}cursive {\bf D}istributed Estimati{\bf O}n~($\mathcal{CREDO}$) characterized by a frugal communication protocol while guaranteeing provably reasonable performance, which improves the dependence between MSE decay and communication rate to $\Theta\left(\mathcal{C}_{t}^{-2+\zeta}\right)$, for arbitrarily small $\zeta>0$. Specifically, this paper focuses on the above described class of \emph{distributed, recursive} algorithms for estimation of an unknown
vector parameter $\btheta$, where each worker sequentially, in time, observes noisy measurements of low-dimensional linear transformations of $\btheta$.
For this problem,we improve the communication efficiency
of existing distributed recursive estimation methods primarily in the \emph{consensus}+\emph{innovations} and the diffusion frameworks ~\citet{kar2011convergence,cattivelli2010diffusion,bajovic2015distributed,Sayed-LMS,sahu2015distributed,jakoveticxaviermoura-11,chen2014multitask}, which in turn may be adapted to improve the communication efficiency of variants such as~ \citet{nedic2014nonasymptotic,ram2010distributed,braca2008enforcing,Nedic-parameter,Ram-Nedich-Siam,Nedic-opt}.
Our contributions are as follows:\\
\noindent We propose a scheme, namely $\mathcal{CREDO}$, where each node at time~$t$
communicates only with a certain probability
that decays sub linearly to zero in~$t$. That is,
communications are increasingly
sparse,
so that communication cost
scales as $\Theta(t^{\delta})$, where
the rate $\delta\in (1/2, 1)$ is a
tunable parameter. \par
\noindent We show that, despite significantly
lower communication cost,
the proposed method achieves
the best possible $\Theta(1/t)$
rate of MSE decay
over~$t$ time-steps ($t$ also equals to per-worker number of data samples obtained over the $t$ time-steps).
Importantly, this result translates
into significant improvements in
the rate at which MSE decays with
communication cost~$\mathcal{C}_t$ --
namely from~$\Theta(1/\mathcal{C}_t)$
with existing methods to
$\Theta(1/\mathcal{C}_t^{2-\zeta})$ with
the proposed method, where
$\zeta>0$ is arbitrarily small. \par
\noindent We further study
asymptotic normality
and the corresponding asymptotic variance
of the proposed method (that in a sense
relates to the constant in the $\Theta(1/t)$
MSE decay rate).
We characterize and quantify interesting
trade-offs between
the communication cost and the asymptotic
variance of the method.
In particular, we
explicitly quantify
the regime (the range of the communication rate parameter~$\delta$) where
the asymptotic variance
is independent of the network topology
and, at the same time,
communication cost
is strictly sub linear~($\delta<1$).
Numerical examples both on synthetic and real data sets
confirm the significantly improved communication efficiency
of the proposed method.\par
\noindent A key insight behind $\mathcal{CREDO}$ is that it recognizes that inter-node communications can be made (probabilistically) increasingly sparse without sacrificing estimation performance. It can be shown using ideas from stochastic approximation that the weights that each node assigns to its neighboring nodes can be made to decrease with time while keeping the estimator strongly consistent. $\mathcal{CREDO}$ replaces such a deterministic weight $w(t)$ ($t$ being time) with a Bernoulli random variable that equals one with probability $w(t)<1$. Thus, $\mathcal{CREDO}$ is much cheaper to implement as communication takes place only with probability $w(t)$, with $w(t)$ decaying to zero. Despite the adaptive weighting being very different, existence of broad regimes of algorithm parameters are shown where $\mathcal{CREDO}$'s estimation performance matches closely the benchmarks iteration-wise. However, as $\mathcal{CREDO}$ has much fewer communications per iteration, it becomes more communication efficient.\par
\noindent Several new technical tools are developed in this paper to achieve the above results that could be of independent interest. Specifically,
the studied setup requires analysis of \emph{mixed time-scale stochastic approximation} algorithms with \emph{three different time scales}. This setup stands in contrast with the classical single time-scale stochastic approximation, the properties of which are well known. It is also very different from the more commonly studied two time-scale stochastic approximation (see, for instance \citet{Borkar-stochapp}) in which a fast process is coupled with a slower dynamical system. We develop here new technical tools that allow us to handle the case of number of operating time-scales to be three instead of two as in \citet{kar2013distributed} for mixed time-scale stochastic approximation~(described in details later).

%
%
%

The rest of the paper is organized as follows. Section~2
describes the problem that we consider, gives the needed preliminaries on
conventional (centralized) and distributed recursive estimation, and reviews related work.
Section~3 presents the novel $\mathcal{CREDO}$ algorithm
that we propose, while Section~4 states our main results on the
algorithm's performance. Finally, we conclude in Section~5.
Proofs of the main and auxiliary results are relegated to the Appendix.

\section{Problem Setup: Motivation and Preliminaries}
\label{sec:prel}
There are $N$ workers deployed in the network. Every worker $n$ at (discrete) time index $t$ makes a noisy observation $\mathbf{y}_{n}(t)$, a noisy linear function of the parameter $\btheta$, where $\btheta\in\mathbb{R}^{M}$. Formally the observation model for the $n$-th worker is given by,
\begin{align}
\label{eq:sens_ob}
\mathbf{y}_{n}(t)=\mathbf{H}_{n}\btheta+\mathbf{\gamma}_{n}(t),
\end{align}
\noindent where $\mathbf{H}_{n}\in\mathbb{R}^{M_{n}\times M}$ is the sensing matrix, where $M_{n}<M$, $\{\mathbf{y}_{n}(t)\} \in \mathbb{R}^{M_{n}}$ is the observation sequence for the $n$-th worker and $\{\mathbf{\gamma}_{n}(t)\}$ is a zero mean temporally independent and identically distributed~(i.i.d.) noise sequence at the $n$-th worker with nonsingular covariance $\mathbf{\Sigma}_{n}$, where $\mathbf{\Sigma}_{n}\in\mathbb{R}^{M_{n}\times M_{n}}$. The noise processes are independent across different workers. We state an assumption on the noise processes before proceeding further. 
\begin{myassump}{M1}
	\label{m:1}
	\emph{There exists $\epsilon_{1}>0$, such that, for all $n$, $\mathbb{E}_{\btheta}\left[\left\|\gamma_{n}(t)\right\|^{2+\epsilon_{1}}\right]<\infty$.}
\end{myassump}
The above assumption encompasses a general class of noise distributions in the setup. The heterogeneity of the setup is exhibited in terms of the sensing matrix and the noise covariances at the worker nodes. Each worker node is interested in reconstructing the true underlying parameter $\btheta$. We assume a worker node is aware only of its local observation model and hence does not know about the observation matrix and noise processes of other worker nodes.
In this paper, we are interested in \emph{recursive distributed estimators}. By recursive, we mean estimators that, at each node $n$, continuously produce (update) estimates of $\btheta$ at each time $t$, i.e., after each new sample $\mathbf{y}_{n}(t)$ is acquired. By distributed, we restrict attention to those estimators in which each node $n$ in the network, at each time $t$, exchanges its current local estimate of $\btheta$ with its immediate neighbors conforming with a pre-specified communication graph, and assimilates its newly acquired observation $\mathbf{y}_n(t)$.
\subsection{Motivation and Related Work}
\label{subsec:motivation}
\noindent We now briefly review the literature on distributed inference and motivate our algorithm $\mathcal{CREDO}$.  Distributed inference algorithms can be broadly divided into two classes. The first class of distributed inference algorithms proposed in \citet{liu2014distributed,ma2015adding,ma2015partitioning,heinze2016dual,zhang2013divide} require a central master node so as to coordinate as far as assigning sub-tasks to the worker nodes is concerned. There are two reasons as to why such methods do not apply in our setting. Firstly, these setups, in order for the central node to be able to assign sub-tasks, require the central node to have access to the entire dataset. However, in the setup considered in this paper, where the data samples are intrinsically distributed among the worker nodes and rather ad-hoc, the presence of a central master node is highly impractical. Even in the case when the data is distributed among nodes to start with, the local data samples collected via \eqref{eq:sens_ob} are not sufficient to uniquely reconstruct the global parameter of interest. In particular, the sensing matrix $\mathbf{H}_{n}$ at an agent $n$ is rank deficient, i.e., $rank(\mathbf{H}_{n}) = M_{n}< M$, in general.  We refer to this phenomenon as \emph{local unobservability}. With communication being the most power hungry aspect for an ad-hoc sensing entity, communicating raw data back to a central node so as to re-assign the data among worker nodes is prohibitive. Thus in such an ad-hoc and distributed setup, a communication protocol should involve information fusion via exchange of the latest estimates among worker nodes thus enabling each worker node to aggregate information about all the entries of the parameter. Secondly, in general, they are not applicable to the heterogeneous sensing model \eqref{eq:sens_ob} being considered here. For example, if $\mathbf{H}_{n} = h\mathbf{I}$, it reduces to the case, where each worker can work independently to obtain a reasonably good estimate of $\btheta$ and algorithms such as $CoCoA^{+}$~(\citet{ma2015adding}) and $Dual-LOCO$~(\citet{heinze2016dual}) may then address the problem efficiently through data splitting across samples and features respectively. However, if $\mathbf{H}_{n} = \mathbf{e}_{n}^{\top}$, where $ \mathbf{e}_{n}$ is the $n$-th canonical basis vector of $\mathbb{R}^{M}$, a random splitting across samples would lead to estimates with a high mean square error, while a feature wise splitting is still possible. But, in the case when, $\mathbf{H}_{n} = \left(\mathbf{e}_{n}+\mathbf{e}_{n-1}\right)^{\top}$, neither sample splitting nor feature splitting is possible and such a setup necessitates more rounds of communication as opposed to just one round of communication at the end as in the case of $CoCoA^{+}$~(\citet{ma2015adding}) and $Dual-LOCO$~(\citet{heinze2016dual}). \\
\noindent The second class of distributed inference algorithms involve setups, which are characterized by the absence of a master node. Communication efficient distributed recursive algorithms in the context of distributed optimization with no central node, where data is available apriori and is not collected in a streaming fashion has been addressed in
\citet{tsianos2012communication,tsianos2013networked,jakovetic2016distributed} through increasingly sparse communication, adaptive communication scheme and selective activation of nodes respectively. However, the explicit characterization of the performance metric, for instance MSE, in terms of the communication cost has not been addressed in the aforementioned references.\\
\noindent The well studied class of distributed estimation algorithms in the \emph{consensus}+\emph{innovations} framework~\citet{kar2011convergence,kar2013distributed} characterize the algorithm parameters, under which estimate sequences optimal in the sense of asymptotic covariance can be obtained. However, the inter-agent message passing and the associated communication cost is not taken into account in the aforementioned algorithms. The lack of exploration into the dimension of communication cost in the context of distributed estimation algorithms in the \emph{consensus}+\emph{innovations} framework motivated us to develop a stochastic communication protocol in this paper, that exploits the redundancy in inter-agent message passing while not compromising on the optimality aspect of the estimate sequence. Hence, in order to test the efficacy of our stochastic message-passing protocol, we take the distributed estimation algorithm proposed in ~\citet{kar2011convergence,kar2013distributed} as the primary benchmark.
\noindent

\subsection{Preliminaries: Oracle and Distributed Estimation}
\label{subsec:prel}
In this section we go over the preliminaries of oracle and distributed estimation.\\
\textbf{Oracle Estimation:}\\
In the setup described above in \eqref{eq:sens_ob}, if a hypothetical oracle node having access to the data samples at all the nodes at all times were to conduct the parameter estimation in an iterative manner, it would do so in the following way:
\begin{align*}
&\mathbf{x}_{c}(t+1)=\mathbf{x}_{c}(t)\nonumber\\&+\underbrace{\frac{a}{t+1}\sum_{n=1}^{N}\mathbf{H}_{n}^{\top}\mathbf{\Sigma}_{n}^{-1}\left(\mathbf{y}_{n}(t)-\mathbf{H}_{n}\mathbf{x}_{c}(t)\right)}_{\text{Global Innovation}},
\end{align*}
where $a$ is a positive constant. It is well known from standard stochastic approximation results (see, for example,~\citet{Nevelson}) that the sequence $\{\mathbf{x}_{c}(t)\}$ generated from the update above converges almost surely to the true parameter $\btheta$. Moreover, the sequence $\{\mathbf{x}_{c}(t)\}$ is asymptotically normal, i.e,
\begin{align*}
\sqrt{t+1}\left(\mathbf{x}_{c}(t)-\btheta\right)\overset{\mathcal{D}}{\Longrightarrow}\mathcal{N}\left(0,\left(N\mathbf{\Gamma}\right)^{-1}\right),
\end{align*}
where $\mathbf{\Gamma}=\frac{1}{N}\sum_{n=1}^{N}\mathbf{H}_{n}^{\top}\mathbf{\Sigma}_{n}^{-1}\mathbf{H}_{n}$. The above established asymptotic normality also points to the conclusion that the mean square error~(MSE) decays as $\Theta(1/t)$.\\
However, such an oracle based scheme may not be implementable in our distributed multi-worker setting with time-varying sparse inter-worker interaction primarily due to the fact that the desired global innovation computation requires instantaneous access to the entire set of network sensed data at all times at the oracle.\\
\textbf{Distributed Estimation:}\\
Distributed estimation scenarios where the global model information is not available at each worker, makes it necessary to communicate at a properly crafted rate. An aptly chosen communication rate would then ensure information flow among the worker nodes so that every worker is able to estimate the parameter of interest. If in the case of a distributed setup, a worker $n$ in the network were to replicate the centralized update by replacing the global innovation in accordance with its local innovation (i.e., based on its local sensed data only), the updates for the parameter estimate becomes
\begin{align*}
&\widehat{\mathbf{x}}_{n}(t+1)=\widehat{\mathbf{x}}_{n}(t)\nonumber\\&+\underbrace{\frac{a}{t+1}\mathbf{H}_{n}^{\top}\left(\widehat{\mathbf{x}}_{n}(t)\right)\mathbf{\Sigma}_{n}^{-1}\left(\mathbf{y}_{n}(t)-\mathbf{H}_{n}\widehat{\mathbf{x}}_{n}(t)\right)}_{\text{Local Innovation}},
\end{align*}
where $\left\{\widehat{\mathbf{x}}_{n}(t)\right\}$ represents the estimate sequence at worker $n$. The above update involves purely decentralized and independent local processing with no collaboration among the workers whatsoever. However, note that in the case when the data samples obtained at each worker lacks information about all the features, the parameter estimates would be erroneous and sub-optimal.
Hence, as a surrogate to the global innovation in the centralized recursions, the local estimators compute a local innovation based on the locally sensed data as a worker has access to the information in its neighborhood. The information loss at a node is compensated by incorporating an agreement or consensus potential into their updates which is then incorporated~(see, for example \citet{kar2008distributed,kar2011convergence,nedic2014nonasymptotic,ram2010distributed,braca2008enforcing,cattivelli2010diffusion,Sayed-LMS,Nedic-parameter,jakoveticxaviermoura-11,Ram-Nedich-Siam,Nedic-opt}) as follows:
\begin{align}
\label{eq:benchmark_ci}
&\mathbf{x}_{n}(t+1) = \mathbf{x}_{n}(t)-\underbrace{\frac{b}{(t+1)^{\delta_{1}}}\sum_{l\in\Omega_{n}(t)}\left(\mathbf{x}_{n}(t)-\mathbf{x}_{l}(t)\right)}_{\text{Neighborhood Consensus}}\nonumber\\&+\underbrace{\frac{a}{t+1}\mathbf{H}_{n}^{\top}\mathbf{\Sigma}_{n}^{-1}\left(\mathbf{y}_{n}(t)-\mathbf{H}_{n}\mathbf{x}_{n}(t)\right)}_{\text{Local Innovation}},
\end{align}
where $0 <\delta_{1}<1$, $\Omega_{n}(t)$ represents the neighborhood of worker $n$ at time $t$ and $a,b$ are appropriately chosen positive constants. In the above scheme, the information exchange among worker nodes is limited to the parameter estimates. It has been shown in previous work that under appropriate conditions~(see, for example \citet{kar2013distributed}), the estimate sequence  $\{\mathbf{x}_{n}(t)\}$ converges to $\btheta$ and is asymptotically normal, i.e.,
\begin{align*}
\sqrt{t+1}\left(\mathbf{x}_{n}(t)-\btheta\right)\overset{\mathcal{D}}{\Longrightarrow}\mathcal{N}\left(0,\left(N\mathbf{\Gamma}\right)^{-1}\right),
\end{align*}
where $\mathbf{\Gamma}=\frac{1}{N}\sum_{n=1}^{N}\mathbf{H}_{n}^{\top}\mathbf{\Sigma}_{n}^{-1}\mathbf{H}_{n}$. The above established asymptotic normality also points to the conclusion that the MSE decays as $\Theta(1/t)$.\\
\textbf{Communication Efficiency}\\
Define the communication cost $\mathcal{C}_{t}$ to be the expected per-node number of transmissions up to iteration $t$, i.e.,
\begin{align}
\label{eq:comm_cost_1}
\mathcal{C}_{t}= \mathbb{E}\left[\sum_{s=0}^{t-1}\mathbb{I}_{\{\textit{node}~C~\textit{transmits~at}~s\}}\right],
\end{align}
where $\mathbb{I}_{A}$ represents the indicator of event $A$. The communication cost $\mathcal{C}_{t}$ for both the oracle estimator and the distributed estimators in \citet{kar2011convergence,cattivelli2010diffusion,Sayed-LMS,chen2014multitask} comes out to be $\mathcal{C}_{t} = \Theta\left(t\right)$, where we note that the time index $t$ also matches the number of per node samples collected till time $t$.  Other close variants of the above mentioned recursive distributed estimation schemes such as the ones in  \citet{nedic2014nonasymptotic,ram2010distributed,braca2008enforcing,Nedic-parameter,jakoveticxaviermoura-11,Ram-Nedich-Siam,Nedic-opt} have a $\mathcal{C}_{t} = \Theta\left(t\right)$ communication cost as well. In other words, we have MSE decaying as $\Theta\left(\frac{1}{\mathcal{C}_{t}}\right)$. Both the paradigms achieve an order-optimal MSE decay rate~$\Theta(1/t)$ in terms of the number
of observations~$t$. Hence, in the setting that we consider, the $\Theta(1/t)$ MSE decay rate with respect to the number of observations cannot be improved upon.\\
In this paper, we ask the highly nontrivial question whether the rate $\Theta(1/\mathcal{C}_t)$ can be improved within the class of recursive distributed estimators. To be particular, we consider recursive distributed estimators with randomized communication protocols, in general. For such estimators, we denote $\mathcal{C}_t$ to be the expected per-node communication cost up to time~$t$. Define the \emph{MSE sensing rate} as the rate at which the MSE decays with the number of per-node samples~$t$. (For example, for  estimator~\eqref{eq:benchmark_ci}, we have
that $\mathrm{MSE}=\Theta(1/t)$.) Similarly, define the \emph{MSE communication rate} as the rate at which the MSE decays with the expected number of per-node communications~$C_t$.
(For example, with estimator~\eqref{eq:benchmark_ci}, we have that $\mathrm{MSE}=\Theta(1/\mathcal{C}_t)$). We are then interested in the achievable pairs (sensing rate, communication rate)
with distributed recursive estimators. A specific question is the following: Given that the $\Theta(1/t)$ sensing rate  cannot be improved in the setting we consider (in fact, limited by the law of large numbers assuming non-degenerate noise covariances), can we improve the communication rate without compromising the sensing rate? If the answer to the above question is affirmative, what specific communication rates are achievable? Subsequent sections provide a detailed study to respond to these questions.\\
\section{$\mathcal{CREDO}$: A communication efficient distributed recursive estimator}
\noindent We now present the proposed $\mathcal{CREDO}$ estimator. $\mathcal{CREDO}$ is based on a specifically handcrafted time decaying communication rate protocol. Intuitively, we basically exploit the idea that, once the information flow starts in the graph and a worker node is able to accumulate sufficient information about the parameter of interest, the need to communicate with its neighboring nodes goes down. Technically speaking, for each node $n$, at every time $t$, we introduce a binary random variable $\psi_{n,t}$, where
\begin{align}
\label{eq:dec_rule_lin_1}
\psi_{n,t}=
\begin{cases}
\rho_{t} &~~\textit{with~probability}~\zeta_{t}\\
0 & ~~\textit{else},
\end{cases}
\end{align}
where $\psi_{i,t}$'s are independent both across time and the nodes, i.e., across $t$ and $n$ respectively. The random variable $\psi_{n,t}$ abstracts out the decision of the node $n$ at time $t$ whether to participate in the neighborhood information exchange or not. We specifically take $\rho_{t}$ and $\zeta_{t}$ of the form
\begin{align}
\label{eq:time_decay}
\rho_{t} = \frac{\rho_{0}}{(t+1)^{\epsilon/2}},~~ \zeta_{t} = \frac{\zeta_{0}}{(t+1)^{(\tau_{1}/2-\epsilon/2)}},
\end{align}
where $0<\epsilon<\tau_{1}$ and $0<\tau_{1}\leq 1$. Furthermore, define $\beta_{t}$ to be
\begin{align}
\label{eq:beta}
\beta_{t}=\left(\rho_{t}\zeta_{t}\right)^{2} = \frac{\beta_{0}}{(t+1)^{\tau_{1}}}.
\end{align}
The pres-specified (possibly sparse) inter-node communication network to which the information exchange between nodes conforms to is modeled as an \emph{undirected} simple connected graph $G=(V,E)$, with $V=\left[1\cdots N\right]$ and~$E$ denoting the set of nodes and communication links. The neighborhood of node~$n$ is given by $\Omega_{n}=\left\{l\in V\,|\,(n,l)\in E\right\}$. The node~$n$ has degree $d_{n}=|\Omega_{n}|$. The structure of the graph is described by the  $N\times N$ adjacency matrix, $\mathbf{A}=\mathbf{A}^\top=\left[\mathbf{A}_{nl}\right]$, $\mathbf{A}_{nl}=1$, if $(n,l)\in E$, $\mathbf{A}_{nl}=0$, otherwise. The graph Laplacian $\mathbf{L}=\mathbf{D}-\mathbf{A}$ is positive definite, with eigenvalues ordered as $0=\lambda_{1}(\mathbf{L}) \leq\lambda_{2}(\mathbf{L}) \leq\cdots \leq \lambda_{N}(\mathbf{L})$, where $\mathbf{D}$ is given by $\mathbf{D}=\mbox{diag}\left(d_{1}\cdots d_{N}\right)$. Moreover, for a connected graph, $\lambda_{2}(\mathbf{L})>0$.
\noindent With the above development in place, we define the random time-varying Laplacian $\mathbf{L}(t)$, where $\mathbf{L}(t)\in\mathbb{R}^{N\times N}$ which abstracts the inter-node information exchange as follows:
\begin{align}
\label{eq:laplacian}
\mathbf{L}_{i,j}(t)=
\begin{cases}
-\psi_{i,t}\psi_{j,t} & \{i,j\}\in E, i\neq j\\
0 & i\neq j, \{i,j\}\notin E\\
-\sum_{l\neq i}\psi_{i,t}\psi_{l,t}& i=j.
\end{cases}
\end{align}
The above communication protocol allows two nodes to communicate only when the link is established in a bi-directional fashion and hence avoids directed graphs. The design of the communication protocol as depicted in \eqref{eq:dec_rule_lin_1}-\eqref{eq:laplacian} not only decays the weight assigned to the links over time but also decays the probability of the existence of a link. Such a design is consistent with frameworks where the working nodes have finite power and hence not only the number of communications, but also, the quality of the communication decays over time.
\noindent We have, for $\{i,j\}\in E$:
\begin{align}
\label{eq:expectation}
&\mathbb{E}\left[\mathbf{L}_{i,j}(t)\right]= -\left(\rho_{t}\zeta_{t}\right)^{2} = -\beta_{t} = -\frac{c_{3}}{(t+1)^{\tau_{1}}}\nonumber\\
&\mathbb{E}\left[\mathbf{L}_{i,j}^{2}(t)\right] = \left(\rho_{t}^{2}\zeta_{t}\right)^{2} = \frac{c_{4}}{(t+1)^{\tau_{1}+\epsilon}}.
\end{align}
\noindent Thus, we have that, the variance of $\mathbf{L}_{i,j}(t)$ is given by,
\begin{align}
\label{eq:variance}
\textit{Var}\left(\mathbf{L}_{i,j}(t)\right) = \frac{\beta_{0}\rho_{0}^{2}}{(t+1)^{\tau_{1}+\epsilon}} - \frac{a^{2}}{(t+1)^{2\tau_{1}}}.
\end{align}
\noindent Define, the mean of the random time-varying Laplacian sequence $\{\mathbf{L}(t)\}$ as $\overline{\mathbf{L}}(t) = \mathbb{E}\left[\mathbf{L}(t)\right]$ and $\widetilde{\mathbf{L}}(t) = \mathbf{L}(t)-\overline{\mathbf{L}}(t)$.
\noindent Note that, $\mathbb{E}\left[\widetilde{\mathbf{L}}(t)\right] = \mathbf{0}$, and
\begin{align}
\label{eq:laplace_res}
\mathbb{E}\left[\left\|\widetilde{\mathbf{L}}(t)\right\|^{2}\right] \leq N^{2}\mathbb{E}\left[\widetilde{\mathbf{L}}_{i,j}^{2}(t)\right] = \frac{N^{2}\beta_{0}\rho_{0}^{2}}{(t+1)^{\tau_{1}+\epsilon}} - \frac{N^{2}a^{2}}{(t+1)^{2\tau_{1}}},
\end{align}
where $\left\|\cdot\right\|$ denotes the $L_{2}$ norm. The above equation follows from the relationship between the $L_{2}$ and Frobenius norms.\\
\noindent We also have that, $\overline{\mathbf{L}}(t)=\beta_{t}\overline{\mathbf{L}}$, where
\begin{align}
	\label{eq:laplacian1}
	\overline{\mathbf{L}}_{i,j}=
	\begin{cases}
	-1 & \{i,j\}\in E, i\neq j\\
	0 & i\neq j, \{i,j\}\notin E\\
	-\sum_{l\neq i}L_{i,l}& i=j.
	\end{cases}
	\end{align}
\noindent We formalize the assumptions on the inter-worker communication graph and global observability.
\begin{myassump}{M2}
	\label{m:2}
	\emph{We require the following global observability condition. The matrix $\mathbf{G}=\sum_{n=1}^{N}\mathbf{H}_{n}^{\top}\mathbf{\Sigma}_{n}^{-1}\mathbf{H}_{n}$ is full rank}.
\end{myassump}
Assumption \ref{m:2} is crucial for our distributed setup. This notion of rendering the parameter locally unobservable while it being globally observable in the context of distributed inference was introduced in \citet{kar2011convergence}, and has been subsequently used in \citet{lalitha2014social,sahu2017recursive}. It is to be noted that such an assumption is needed for even a setup with a centralized node which has access to all the data samples at each of the worker nodes at each time. Assumption \ref{m:2} ensures that if a node could stack all the data samples together at any time $t$, it would have sufficient information about the parameter of interest so as to be able to estimate the parameter of interest without any communication. Hence, the requirement for this assumption naturally extends to our distributed setup.
\begin{myassump}{M3}
	\label{m:3}
	\emph{The inter-worker communication graph is connected on average, i.e., $\lambda_{2}(\overline{\mathbf{L}}) > 0$, which implies $\lambda_{2}(\overline{\mathbf{L}}(t))>0$, where $\overline{\mathbf{L}}(t)$ denotes the mean of the Laplacian matrix $\mathbf{L}(t)$ and $\lambda_{2}\left(\cdot\right)$ denotes the second smallest eigenvalue}.
\end{myassump}
Assumption \ref{m:3} ensures consistent information flow among the worker nodes. Technically speaking, the communication graph modeled here as a random undirected graph need not be connected at all times. Hence, at any given time, only a few of the possible links could be active. The connectedness in average basically ensures that over time, the information from each worker node in the graph reaches other worker nodes over time in a symmetric fashion and thus ensuring information flow. It is to be noted that assumption \ref{m:3} ensures that $\overline{\mathbf{L}}(t)$ is connected at all times as $\overline{\mathbf{L}}(t)=\beta_{t}\overline{\mathbf{L}}$.
With the communication protocol established, we propose an update, where every node $n$ generates an estimate sequence $\{\mathbf{x}_{n}(t)\}$, with $\mathbf{x}_{n}(t)\in\mathbb{R}^{M}$, in the following way:
\begin{align}
\label{eq:ci_update_worker}
&\mathbf{x}_{n}(t+1) = \mathbf{x}_{n}(t)-\underbrace{\sum_{l\in\Omega_{n}}\psi_{n,t}\psi_{l,t}\left(\mathbf{x}_{n}(t)-\mathbf{x}_{l}(t)\right)}_{\text{Neighborhood Consensus}}\nonumber\\&+\underbrace{\alpha_{t}\mathbf{H}_{n}^{\top}\mathbf{\Sigma}_{n}^{-1}\left(\mathbf{y}_{n}(t)-\mathbf{H}_{n}\mathbf{x}_{n}(t)\right)}_{\text{Local Innovation}},
\end{align}
where $\Omega_{n}$ denotes the neighborhood of node $n$ with respect to the network induced by $\overline{\mathbf{L}}$ and $\alpha_{t}$ is the innovation gain sequence which is given by $\alpha_t = a/(t+1)$. It is to be noted that a node $n$ can send and receive information in its neighborhood at time $t$, when $\psi_{n,t}\neq 0$. At the same time, when $\psi_{n,t}= 0$, node $n$ neither transmits nor receives information. The link between node $n$ and node $l$ gets assigned a weight of $\rho_{t}^{2}$ if and only if $\psi_{n,t}\neq 0$ and $\psi_{l,t}\neq 0$.
\begin{Remark}
\label{rm:1}
The stochastic update procedure~\eqref{eq:ci_update_worker}, employed here may be viewed as a mixed time-scale stochastic approximation procedure as opposed to the classical single time-scale stochastic approximation, the properties of which are well known. Note, the above notion of mixed time-scale is very different from the more commonly studied two time-scale stochastic approximation (see, for instance \citet{Borkar-stochapp}) in which a fast process is coupled with a slower dynamical system. More relevant to our study are the mixed time-scale dynamics encountered in \citet{Gelfand-Mitter} and \citet{kar2013distributed} in which a single update procedure is influenced by multiple potentials with different time-decaying weights.
However, as opposed to the innovations term being a martingale difference sequence in the context of mixed time-scale stochastic approximation as proposed in \citet{Gelfand-Mitter}, the mixed time-scale stochastic approximation employed in this paper does not have an innovation term which is a martingale difference sequence and hence is of sufficient technical interest. The addition of the residual Laplacian $\widetilde{L}(t)$ sequence in the update further complicates the update in the context of this paper, by making the number of operating time-scales to be three instead of two as in \citet{kar2013distributed} for which we had to develop new technical machinery.
\end{Remark}
The above update can be written in a compact form as follows:
\begin{align}
\label{eq:ci_update}
&\mathbf{x}(t+1)=\left(\mathbf{I}_{NM}-\mathbf{L}(t)\otimes\mathbf{I}_{M}\right)\mathbf{x}(t)\nonumber\\&+\alpha_{t}\mathbf{G}_{H}\mathbf{\Sigma}^{-1}\left(\mathbf{y}(t)-\mathbf{G}_{H}^{\top}\mathbf{x}(t)\right),
\end{align}
where $\alpha_{t}=\frac{a}{t+1}$, $\mathbf{x}(t)=[\mathbf{x}_{1}^{\top}(t) ~\mathbf{x}_{2}^{\top}(t) \cdots \mathbf{x}_{N}^{\top}(t)]^{\top}$, $\mathbf{G}_{H}=diag[\mathbf{H}_{1}^{\top}, \mathbf{H}_{2}^{\top},\cdots, \mathbf{H}_{N}^{\top}]$, $\mathbf{y}(t)=[\mathbf{y}_{1}^{\top}(t) ~\mathbf{y}_{2}^{\top}(t) \cdots \mathbf{y}_{N}^{\top}(t)]^{\top}$ and $\mathbf{\Sigma}=diag\left[\mathbf{\Sigma}_{1},\cdots,\mathbf{\Sigma}_{N}\right]$.
\begin{Remark}
\label{rm:2}
The Laplacian sequence that plays a role in the analysis in this paper, takes the form $L(t)=\beta_{t}\overline{L}+\widetilde{L}(t)$, where $\widetilde{L}(t)$, the residual Laplacian sequence, does not scale with $\beta_{t}$ owing to the fact that the communication rate is chosen adaptively making the analysis significantly different from \citet{kar2013distributed}. Thus, unlike \citet{kar2013distributed}, the Laplacian matrix sequence is not identically distributed; the sequence of effective Laplacians have a decaying mean, thus adding another time-scale in the already mixed time-scale dynamics which necessitates the development of new technical tools to establish the order optimal convergence of the estimate sequence.
\end{Remark}
We formalize an assumption on the innovation gain sequence $\{\alpha_{t}\}$ before proceeding further.
\begin{myassump}{M4}
	\label{m:4}
	\emph{Let $\lambda_{min}\left(\cdot\right)$ denote the smallest eigenvalue. We require that $a$ satisfies\footnote{Note that, as will be shown later, $\mathbf{\Gamma}$ and $\overline{\mathbf{L}}\otimes\mathbf{I}_{M}+\mathbf{G}_{H}\mathbf{\Sigma}^{-1}\mathbf{G}_{H}^{\top}$ are positive definite matrices under the stated assumptions.}, $a\min\{\lambda_{min}\left(\mathbf{\Gamma}\right),\lambda_{min}\left(\overline{\mathbf{L}}\otimes\mathbf{I}_{M}+\mathbf{G}_{H}\mathbf{\Sigma}^{-1}\mathbf{G}_{H}^{\top}\right),\beta_{0}^{-1}\}\ge 1$, where $\otimes$ denotes the Kronecker product.}
\end{myassump}
The communication cost per node for the proposed algorithm is given by $\mathcal{C}_{t} = \sum_{s=0}^{t-1}\zeta_{s} = \Theta\left(t^{1+(\epsilon-\tau_{1})/2}\right)$,
which in turn is strictly sub-linear as $\epsilon < \tau_{1}$.
\section{Main Results}
\label{sec:main_res}
In this section, we present the main results of the proposed algorithm $\mathcal{CREDO}$, while the proof of the main results are relegated to the Appendix. The first result concerns with the consistency of the estimate sequence $\{\mathbf{x}_{n}(t)\}$.
\begin{Theorem}
	\label{th:cons} Let assumptions~\ref{m:1}-\ref{m:4} hold and let $\tau_{1}$ in the consensus potential in \eqref{eq:beta} be such that $0<\tau_{1}\leq 1$. Consider the sequence $\{\mathbf{x}_{n}(t)\}$ generated by \eqref{eq:ci_update_worker} at each worker $n$. Then, for each $n$, we have
	\begin{equation}
	\label{eq:th_cons1}
	\mathbb{P}_{\btheta}\left(\lim_{t\rightarrow\infty}\mathbf{x}_{n}(t)=\btheta\right)=1.
	\end{equation}
	In particular, if $\tau_{1}$ satisfies $0<\tau_{1}\leq 0.5-(2+\epsilon_{1})^{-1}$, we have that for all $\tau\in [0, 1/2)$, $\mathbb{P}_{\btheta}\left(\lim_{t\rightarrow\infty}(t+1)^{\tau}\|\mathbf{x}_{n}(t)-\btheta\|=0\right)=1$. 
	%
\end{Theorem}
\noindent At this point, the estimate sequence generated by $\mathcal{CREDO}$ at any worker $n$ is strongly consistent, i.e., $\mathbf{x}_{n}(t)\rightarrow\btheta$ almost surely~(a.s.) as $t\rightarrow\infty$. Furthermore, the above characterization for $0<\tau_{1}\leq 0.5-(2+\epsilon_{1})^{-1}$ yields order-optimal convergence, i.e., from results in classical estimation theory, it is known that there exists no $\tau\geq 1/2$ such that a estimator $\{\btheta_{c}(t)\}$ satisfies $(t+1)^{\tau}\|\btheta_{c}(t)-\btheta\|\rightarrow 0$ a.s. as $t\rightarrow\infty$.
We now state a main result of this paper which establishes the MSE communication rate for the proposed algorithm $\mathcal{CREDO}$.
\begin{Theorem}
	\label{th:credo}
	Let the hypothesis of Theorem \ref{th:cons} hold. Then, we have,
	\begin{align}
		\label{eq:credo_1}
		\mathbb{E}_{\btheta}\left[\left\|\mathbf{x}_{n}(t)-\btheta\right\|^{2}\right] = \Theta\left(\mathcal{C}_{t}^{-\frac{2}{\epsilon-\tau_{1}+2}}\right),
		\end{align}
	where $\epsilon < \tau_{1}$ and is as defined in \eqref{eq:time_decay}.
\end{Theorem}
The version of the $\mathcal{CREDO}$ algorithm, with $\beta_{t} = a(t+1)^{-1}$, achieves a communication cost of $\mathcal{C}_{t} = \Theta\left(t^{0.5(1+\epsilon)}\right)$. Hence, the MSE as a function of $\mathcal{C}_{t}$ in the case of $\tau_{1}=1$ is given by $\mbox{MSE} = \Theta(\mathcal{C}_{t}^{-2/(1+\epsilon)})$. However, it can be shown from standard arguments in stochastic approximation that updates with $\beta_{t} = a(t+1)^{-1-\delta}$ with $\delta>0$, though results in a communication cost of $\mathcal{C}_{t} = \Theta(t^{0.5(1+\epsilon-\delta)})$, it does not generate estimate sequences which converge to $\btheta$.\\
\noindent With the above development in place, we state a result which allows us to benchmark the asymptotic efficiency of the proposed algorithm and the instantiations of it in terms of $\tau_{1}$. To be specific, the next result establishes the asymptotic normality of the parameter estimate sequence $\{\mathbf{x}_{n}(t)\}$ and characterizes the asymptotic covariance of the estimate sequence.
\begin{Theorem}
	\label{th:2}
	Let the hypotheses of Theorem \ref{th:cons} hold and in addition let $0<\tau_{1}\leq 0.5-(2+\epsilon_{1})^{-1}$. Then, we have,
	\begin{align}
		\label{eq:l3_1}
		\sqrt{t+1}\left(\mathbf{x}_{n}(t)-\btheta\right)\overset{\mathcal{D}}{\Longrightarrow}\mathcal{N}\left(0,\frac{a\mathbf{I}}{2N}+\frac{\left(\mathbf{\Gamma}-\frac{\mathbf{I}}{2a}\right)^{-1}}{4N}\right),
		\end{align}
	where $\mathbf{\Gamma}=\frac{1}{N}\sum_{n=1}^{N}\mathbf{H}_{n}^{\top}\mathbf{\Sigma}_{n}^{-1}\mathbf{H}_{n}$.
\end{Theorem}
\noindent The asymptotic covariance as established in \eqref{eq:l3_1} is independent of the network. Technically speaking, as long as the averaged Laplacian $\overline{\mathbf{L}}$ is connected, and the consensus and the innovation potentials, i.e., $\beta_{t}$ and $\alpha_{t}$ respectively are chosen appropriately, the asymptotic covariance is independent of the network connectivity, i.e., it is independent of the network instantiations across all times and is just a function of the sensing model parameters and the noise covariance. It is to be noted that the optimal asymptotic covariance achieved by the oracle estimator is given by $N\mathbf{\Gamma}$. Such an asymptotic covariance can be achieved by a distributed setup where every worker node is aware of every other worker node's sensing model. To be particular, if a gain matrix $G = \sum_{n=1}^{N}N^{-1}\mathbf{H}_{n}^{\top}\mathbf{\Sigma}_{n}^{-1}\mathbf{H}_{n}$ is multiplied to the innovation term of the update in \eqref{eq:ci_update_worker}, the optimal asymptotic covariance is achievable~(see, for example \citet{kar2013distributed}). However, such an update would need global model information available at each worker node.\\
\noindent We now discuss the interesting trade-offs between the communication cost and the asymptotic covariance that follow from Theorem \ref{th:2} and some existing results \citet{kar2012distributed,kar2013distributed}~(see Table \ref{tab:1}). At this juncture, we consider the setup, where the $\tau_{1}$ in the consensus potential $\beta_{t}$ in \eqref{eq:beta} is taken to be $1/2-(2+\epsilon_{1})^{-1}\le \tau_{1} \leq 1$. We specifically consider the case where $\tau_{1}=1$. It has been established in prior work~(see, for example \citet{kar2012distributed}) that in this case the asymptotic covariance depends on the network instantiation. To be specific, the averaged Laplacian $\overline{\mathbf{L}}$ which abstracts out the time-averaged information flow among the worker nodes has a key role in the asymptotic covariance in such a case. However, such a scheme, i.e., a single time scale variant\footnote{Single time-scale in the sense that the weights sequences $\beta_{t}$ and $\alpha_{t}$ have the same asymptotic decay rate.} of the proposed algorithm~(in general for $1/2-(2+\epsilon_{1})^{-1}\le \tau_{1} \leq 1$) enjoys a lower communication rate. Technically speaking, for the case when $\tau_{1}=1$, the communication rate is given by $\mathcal{C}_{t}=\Theta\left(t^{0.5(1+\epsilon)}\right)$. Hence, there is an intrinsic trade-off between the communication rate and the achievable asymptotic variance.\\
Intuitively, the algorithm exhibits a threshold behavior in terms of the consensus potential $\tau_{1}$. The threshold behavior is summarized in table \ref{tab:1}.
\begin{table}[h]
	\centering
	\caption{Trade-off between Communication cost and Asymptotic Covariance}\label{tab:1}
	\begin{tabular}{ l  l  l  l }
	{\small\textbf{Trade-Off}} & {\small\textbf{Convergence}}&{\small\textbf{Asymptotic Covariance}}&{\small\textbf{Comm. Cost.}}\\
		\hline \\
		${\bf0<\tau_{1}<\frac{1}{2}-\frac{1}{2+\epsilon_{1}}}$ & \textbf{Consistent}& \textbf{Network Independent} & $\Theta\left(t^{\frac{3}{4}+\frac{\epsilon}{2}}\right)$ \\
		${\bf\frac{1}{2}-\frac{1}{2+\epsilon_{1}}\leq \tau_{1}\leq 1}$ & \textbf{Consistent} & Network Dependent & $\Theta\left(t^{\frac{1+\epsilon}{2}}\right)$ \\
		${\bf\tau_{1}>1}$ & Does not converge & Diverges & ${\bf\Theta\left(1\right)}$ \\
		\hline
	\end{tabular}
\end{table}In the case when, $\tau_{1} < 1/2-\frac{1}{2+\epsilon_{1}}$, the algorithm achieves a network independent asymptotic covariance while ensuring the communication rate to be strictly sub linear. However, in the case when $1/2-\frac{1}{2+\epsilon_{1}}\le \tau_{1} \leq 1$, the algorithm has a communication rate which is lower than the previous regime, but then achieves asymptotic covariance which depends on the network explicitly. Finally, in the case when $\tau_{1}>1$, the algorithm does not even converge to the true underlying parameter.

\section{Simulation Experiments}
This section corroborates our theoretical findings through simulation examples
and demonstrates  on both synthetic and real data sets communication efficiency of $\mathcal{CREDO}$.
Subsection~{5.1} considers synthetic data, while Subsection~{5.2} presents simulation results on real data sets.

\subsection{Synthetic Data}
Specifically, we compare the proposed communication-efficient distributed estimator, $\mathcal{CREDO}$, with the benchmark distributed recursive estimator in \eqref{eq:benchmark_ci} which utilizes all inter-neighbor communications at all times, i.e., has a linear communication cost. The example demonstrates that the proposed communication-efficient estimator matches the MSE rate of the benchmark estimator. The simulation also shows that the proposed estimator improves the MSE \emph{communication rate} with respect to the benchmark.
The simulation setup is as follows. We consider three instances of undirected graphs with $N=20$ nodes, with relative degrees\footnote{Relative degree is the ratio of the number of links in the graph to the number of possible links in the graph.} of nodes slated at $0.3736$, $0.5157$ and $0.6578$. The graphs were generated as connected graph instances of the random geometric graph model with radius $r=\sqrt{\mathrm{ln}(N)/N}$.
We set $M=10$ and $M_n=1$, for all $n=1,...,N$; i.e., the unknown parameter $\btheta\in\mathbb{R}^{10}$, while each node makes a scalar observation at each time~$t$. The noises $\gamma_n(t)$ are Gaussian and are i.i.d. both in time and across nodes and
have the covariance matrix equal to~$0.25 \times I$. The sampling matrices $\mathbf{H}_n$'s are chosen to be $2$-sparse, i.e., every nodes observes a linear combination of two arbitrary entries of the vector parameter. The non-zero entries of the $\mathbf{H}_{n}$'s are sampled from a standard normal distribution. The sampling matrices $\mathbf{H}_n$'s at the same time satisfy Assumption \ref{m:2}.
The parameters of the benchmark and the proposed estimator are as follows. The benchmark estimator's consensus weight is set to $0.1(t+1)^{-0.49}$. With the proposed estimator, we study the first two regimes as illustrated in Table \ref{tab:1}, i.e., $0<\tau_{1}<\frac{1}{2}$ and $\frac{1}{2}\leq \tau_{1}\leq 1$. For the second regime, we study two different cases. We set $\rho_t = 0.1(t+1)^{-0.01}$ for both the regimes. We set $\zeta_t = (t+1)^{-0.235}$, $\zeta_t = (t+1)^{-0.315}$  and $\zeta_t = (t+1)^{-0.49}$ for the above mentioned first and two cases of the second regime respectively; that is, with the proposed estimator, we set $\epsilon=0.01$, $\tau_1 = 0.49$, $\epsilon=0.01$, $\tau_1 = 0.65$ and $\epsilon=0.01$, $\tau_1 = 1$ for the first and two cases of the second regime respectively. Note that the Laplacian matrix associated with the benchmark estimator and the expected Laplacian matrix associated with the proposed estimator, $\mathcal{CREDO}$  are equal in each of the three generated networks, i.e., $\mathbf{\overline{L}} =\mathbf{L}$.
With all the three estimators, the innovation weight is set to $\alpha_t = (3.68(t+20))^{-1}$. Note that all the theoretical results in the paper hold unchanged for the ``time-shifted'' $\alpha_t$ used here.
The purpose of the shift in the innovation weight is to avoid large innovation weights in the initial iterations. As a performance metric, we use
the relative MSE estimate averaged across nodes:
\begin{align*}
\frac{1}{N}\sum_{n=1}^N \frac{ \|\mathbf{x}_n(t)-\btheta \|^2}{ \|\mathbf{x}_n(0)-\btheta\|^2},
\end{align*}
further averaged across $50$ independent runs of the three estimators. Here, $\mathbf{x}_n(0)$ is node $n$'s initial estimate. With both estimators, at each run, at all nodes, we set $\mathbf{x}_n(0)=0$. Figure~\ref{Figure_1} plots the estimated relative MSE versus time~$t$ in log-log scale for the three networks.
From figure \ref{Figure_1}, we can see that the MSE decay of the proposed estimator coincides with that of the benchmark estimator, especially in the $\tau_{1}=0.49$ regime across all the three networks, inspite of having lower communication costs. $\mathcal{CREDO}$ with $\tau_{1}=0.65$ and $\tau_{1}=1$, has higher convergence constants\footnote{It basically points to the fact that, though the MSE in all the cases have a $t^{-1}$ scaling, the variance decays of the $\tau_{1}=0.65$ and $\tau_{1}=1$ cases involve bigger constants and thus larger variances.} with respect to the MSE decay rates as compared to the benchmark estimator, though with far lower communication costs. We can also see that, for network $1$ and network $2$, with relative degree slated at $0.3136$ and $0.5157$ respectively, the MSE in the case of $\tau_{1}=0.65$ and $\tau_{1}=1$ shifts further away from the MSE curve of network $3$ and thus illustrating the network dependent convergence constant in the regime $1/2\le\tau_{1}\leq 1$. At the same time, from Figure \ref{Figure_1} it can be seen that with $\tau_{1}=0.49$, the convergence is practically independent of the network similar to the convergence of the benchmark estimator, as predicted by Theorem \ref{th:2}. Figure \ref{Figure_2} plots the estimated relative MSE versus average per-node communication cost~$C_t$. We can see that the proposed scheme has an improved communication rate with respect to the benchmark, as predicted by the theory. In spite of higher convergence constants with respect to the MSE decay rates, in the case of $\tau_{1}=0.65$ and $\tau_{1}=1$, the MSE decay rate in terms of the communication cost is still faster than the benchmark estimator. Also, in the case of $\tau_{1}=0.49$, there is a close to $10\times$ reduction in the communication cost for the same achievable relative MSE of $0.005$  as compared to the benchmark estimator.
Figure \ref{Figure_2}, illustrates the trade-off between the MSE decay rate and the communication cost, there in, the lowest communication cost enjoyed by $\mathcal{CREDO}$ results in higher convergence constant with respect to the MSE decay, while the lowest convergence constant with respect to the MSE decay rate enjoyed by the benchmark estimator results in the highest communication cost.
\begin{figure}[H]
	\centering
	\includegraphics[angle=-90,origin=c,height=4.0 in,width=4.0 in]{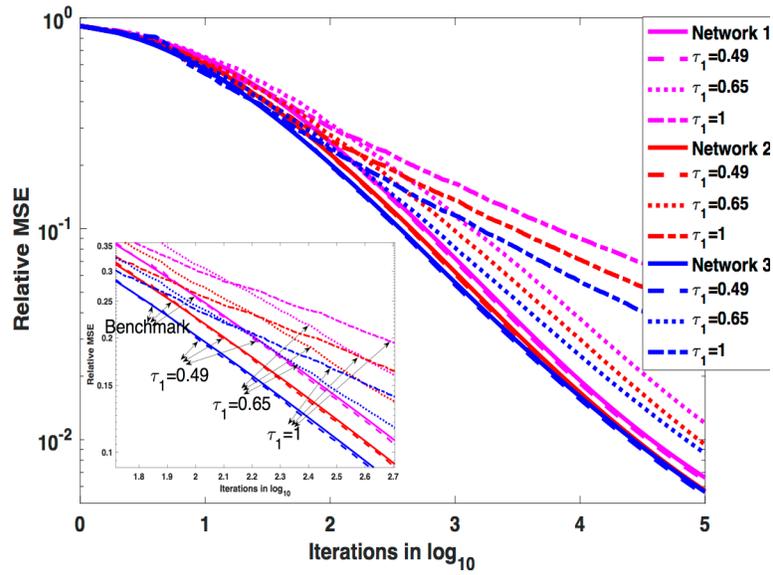}
	\caption{\emph{Comparison of the proposed and benchmark estimators in terms of relative MSE: Number of Iterations}. The solid lines represent the benchmark, the three different colors indicate the three different networks, while the three regimes are represented by the dotted lines.}
	\label{Figure_1}
\end{figure}
\begin{figure}[H]
	\centering
		\includegraphics[angle=-90,origin=c,height=4.0 in,width=4.0 in]{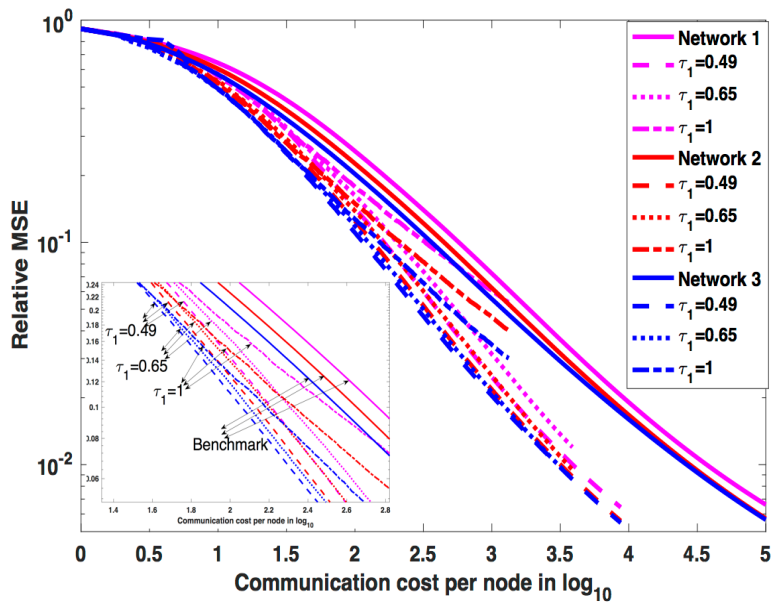}
		\caption{\emph{Comparison of
			the proposed and benchmark estimators in terms of relative MSE: Communication cost per node}. The solid lines represent the benchmark, the three different colors indicate the three different networks, while the three regimes are represented by the dotted lines.}
			\label{Figure_2}
	\end{figure}

\subsection{Real Datasets}
In order to evaluate the performance of $\mathcal{CREDO}$, we ran experiments on three real-world datasets, namely cadata~(\citet{LibSVM}), Abalone~(\citet{Lichman:2013}) and bank~(\citet{Delve}).\\
For the cadata dataset~(20640 data points, 8 features), we divided the samples into $20$ equal parts of $900$ data points each, after keeping $2640$ data points as the test set. For the 20 node network, we constructed a random geometric graph. For the Abalone dataset~(4177 data points, 8 features), we divided the samples into $10$ equal parts of $360$ points each, after keeping $577$ data points as the test set. For the 10 node network, we constructed a random geometric graph. For the bank dataset~(8192 data points, 9 features), we divided the samples into $20$ equal parts equal parts of $350$ points each, after keeping $1192$ data points as the test set. For the 20 node network, we constructed a random geometric graph. We added Gaussian noise to the dependent variables, i.e., housing price, the age of Abalone and fraction of rejecting customers respectively.
  The training datasets, with respect to the sensing model~\eqref{eq:sens_ob}, have dynamic regressors (a regressor here corresponds to a feature vector of one data point), i.e, time-varying $\mathbf{H}_{n}$'s for each agent $n$.
  Thus, we perform a pre-processing step where we average the training data points' regressors at each node to obtain an averaged $\overline{\mathbf{H}}_{n}$, which is then subsequently used at every iteration~$t$ in the update~\eqref{eq:ci_update_worker}. For each experiment (each dataset), a consistency check is done by ensuring that $\sum_{n=1}\overline{\mathbf{H}}_{n}^{\top}\mathbf{\Sigma}_{n}^{-1}\overline{\mathbf{H}}_{n}$ is invertible and thus global observability holds. As the number of data points at each node are the same, we sample
  along iterations~$t$ data points at each node without replacement, and thus the total number of iterations~$t$ we run the algorithms equals the number of data points at each node. In other words, the algorithm passes through each data point exactly once.
   We summarize the comparison of the number of communications needed by $\mathcal{CREDO}$ and the benchmark algorithm at the test error obtained after the total number of iterations in Table \ref{tab:2}. In particular, the test errors obtained in the cadata, abalone and the bank dataset are $0.015$, $0.03$ and $0.007$ of the initial test error, respectively. In figures \ref{fig:cahousing}, \ref{fig:abalone} and \ref{fig:bank}, we plot the evolution of the test error for each of the datasets as a function of the number of iterations and the communication cost. It can be seen that while $\mathcal{CREDO}$ matches the final test error of that of the benchmark algorithm, it requires on average thrice as less number of communications.

Note that the theoretical setup in this paper rigorously establishes results pertaining to observation models with static regressors, i.e., static sensing matrices. However, the simulations on the real world datasets show that in spite of the time-varying regressors, 
 the algorithm continues to demonstrate its improved communication efficiency over the benchmark. Moreover, as the sampling at each node is without replacement, the transients as far the performance is concerned can be improved by making the weight sequences decay after a few iterations instead of every iteration. Such a decay, while ensuring that the algorithm requirements are satisfied, would ensure faster assimilation of new data points in the transient phase.\\
\begin{figure}
	\centering
	\begin{subfigure}{.5\textwidth}
		\centering
		\includegraphics[width=.9\linewidth]{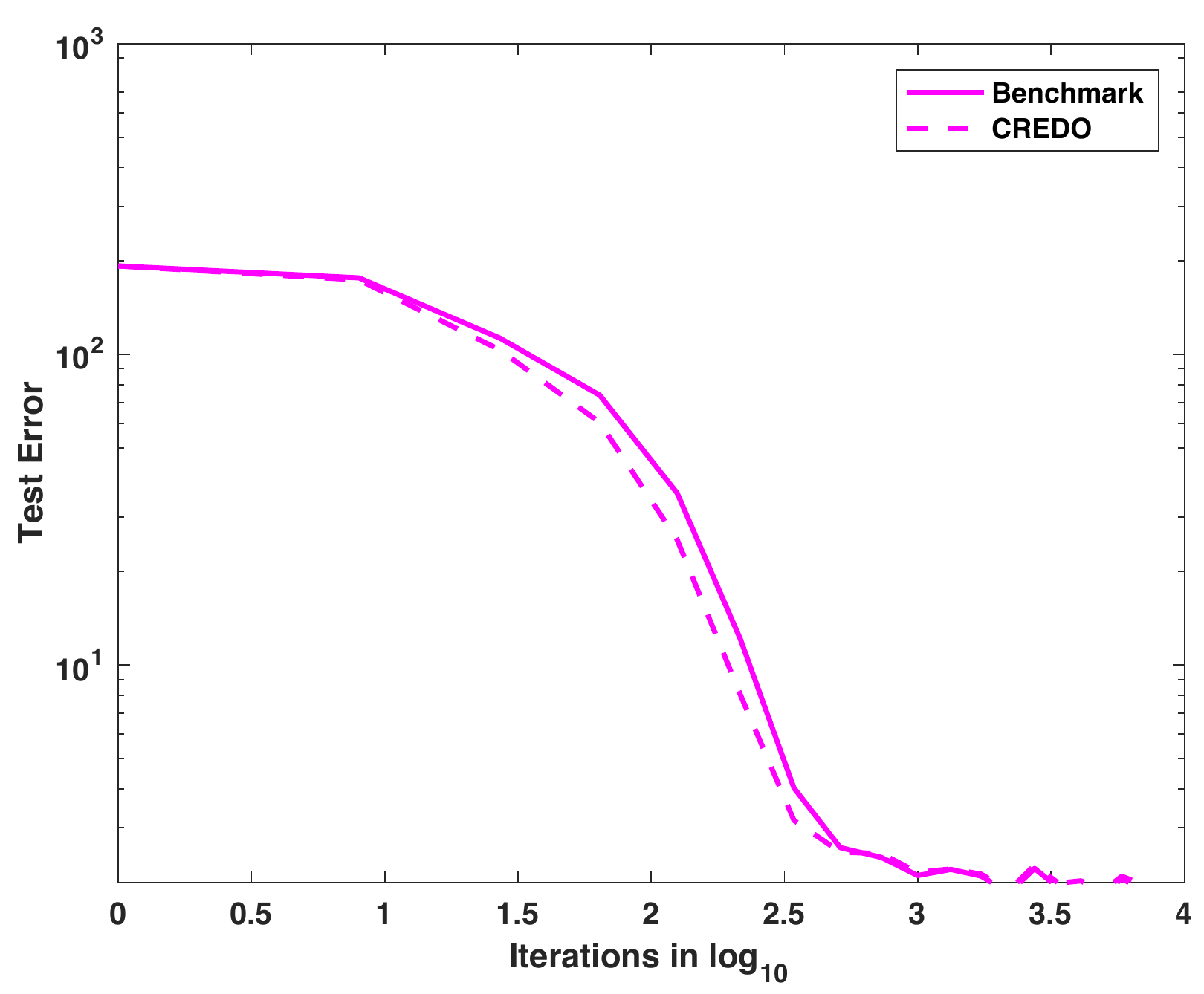}
		\caption{Comparison of Test Error: Number of Iterations}
		\label{fig:cahousing_mse}
	\end{subfigure}%
	\begin{subfigure}{.5\textwidth}
		\centering
		\includegraphics[width=.9\linewidth]{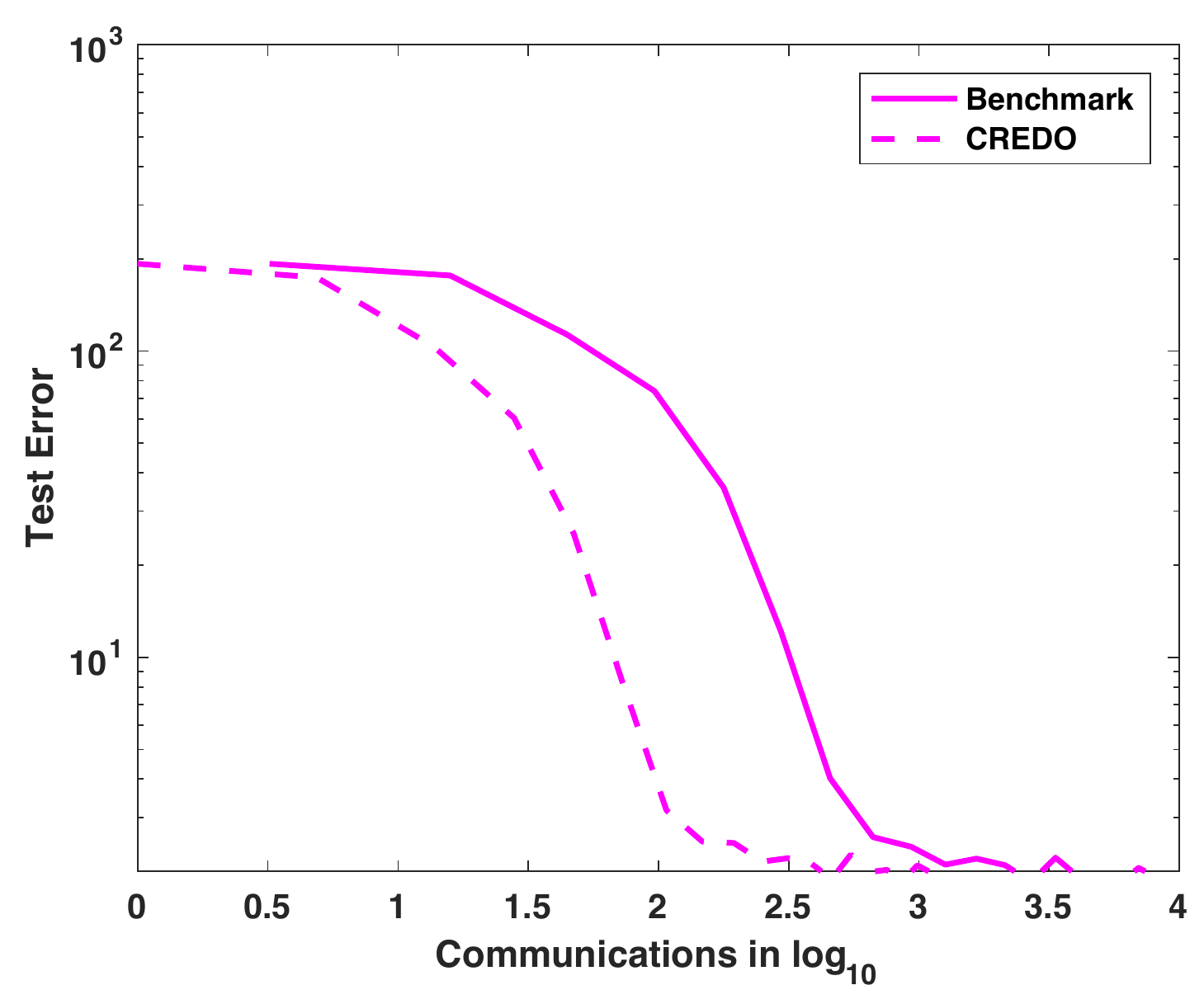}
		\caption{Comparison of Test Error: Communication cost per node}
		\label{fig:cahousing_comm_cost}
	\end{subfigure}
	\caption{CADATA Dataset: Comparison of
		the $\mathcal{CREDO}$ and benchmark estimators}
	\label{fig:cahousing}
\end{figure}
\begin{figure}
	\centering
	\begin{subfigure}{.5\textwidth}
		\centering
		\includegraphics[width=.9\linewidth]{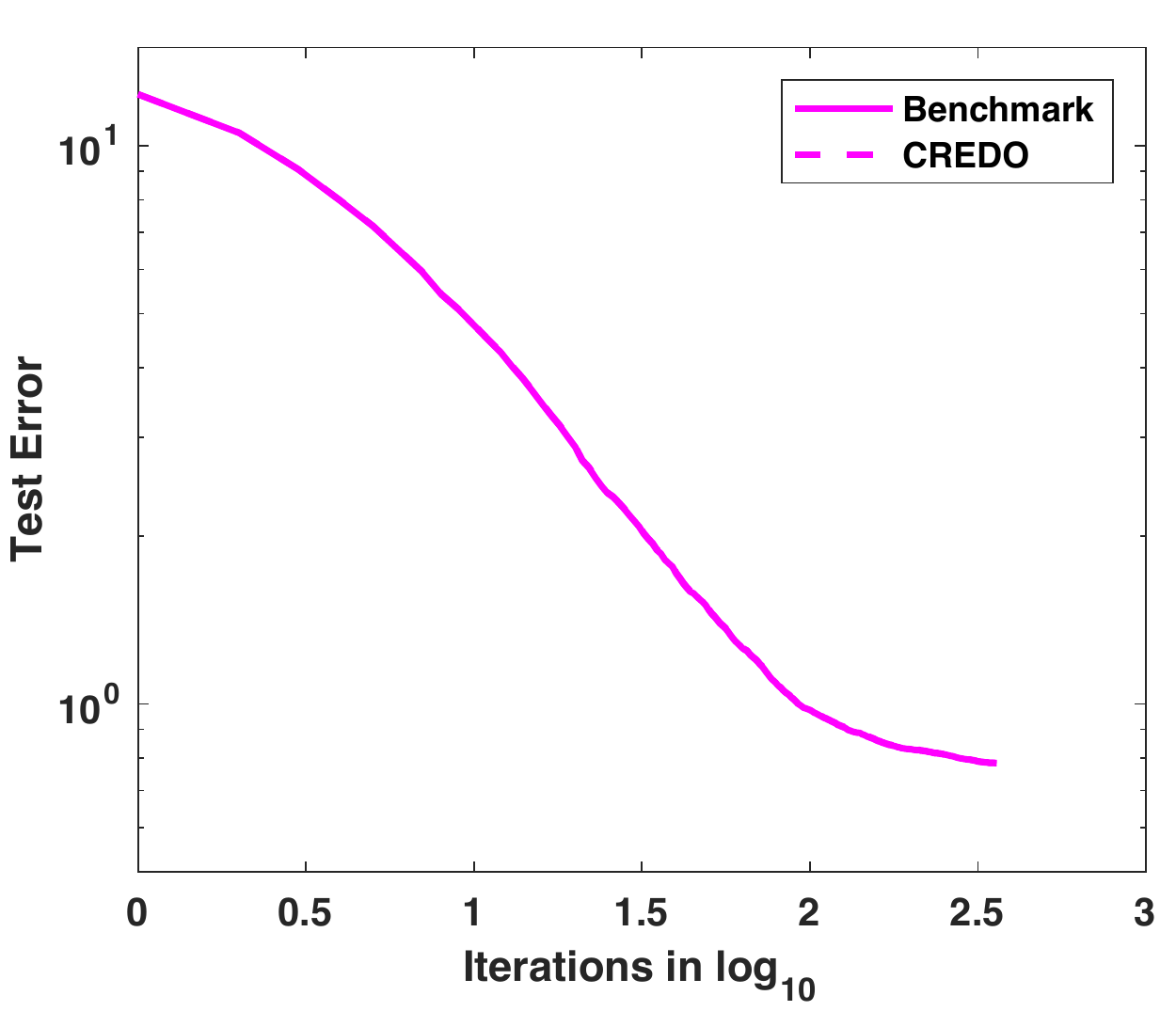}
		\caption{Comparison of Test Error: Number of Iterations}
		\label{fig:abalone_mse}
	\end{subfigure}%
	\begin{subfigure}{.5\textwidth}
		\centering
		\includegraphics[width=.9\linewidth]{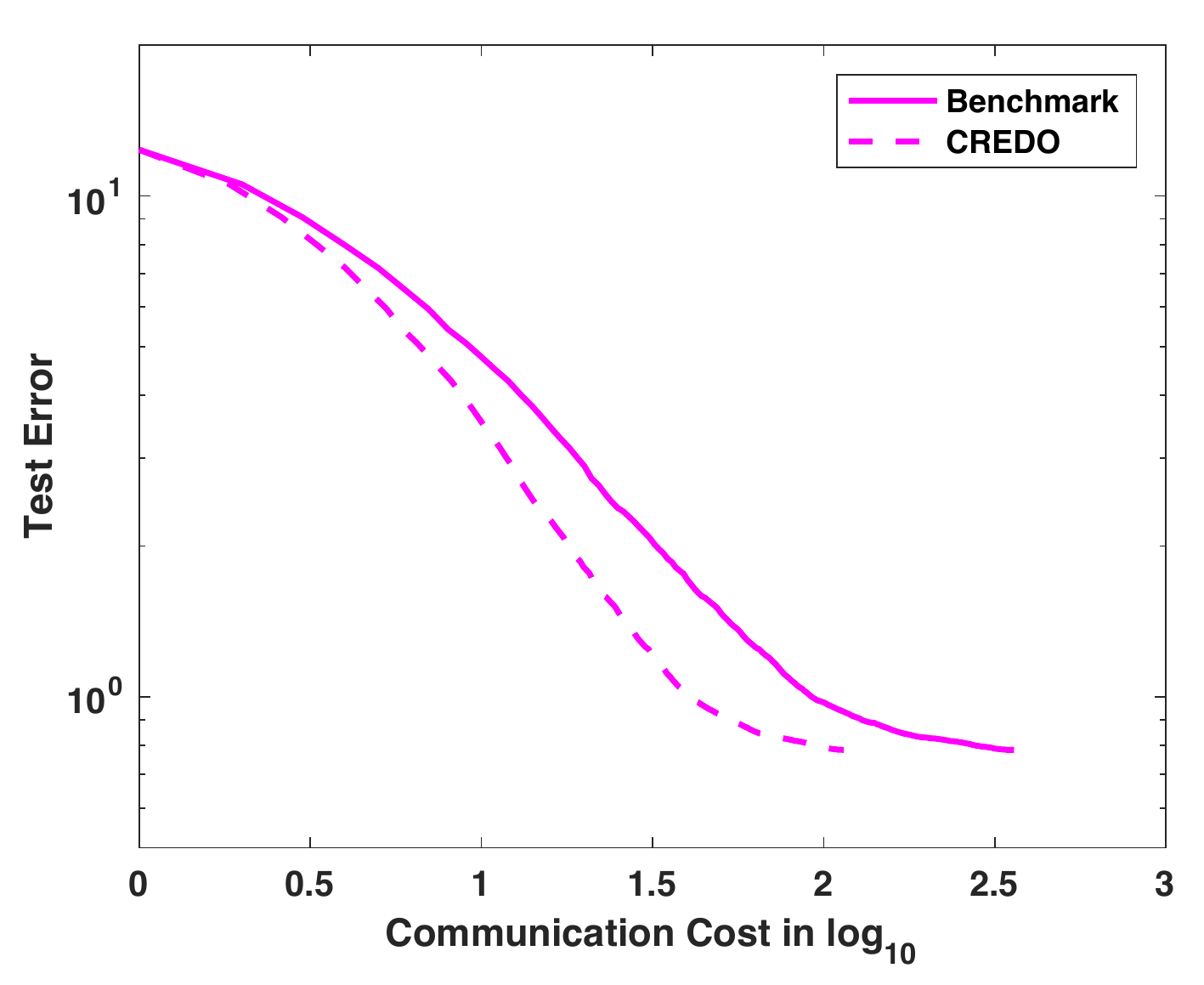}
		\caption{Comparison of Test Error: Communication cost per node}
		\label{fig:abalone_comm_cost}
	\end{subfigure}
	\caption{Abalone Dataset: Comparison of
		the $\mathcal{CREDO}$ and benchmark estimators}
	\label{fig:abalone}
\end{figure}
\begin{figure}
		\centering
		\begin{subfigure}{.5\textwidth}
			\centering
			\includegraphics[width=.9\linewidth]{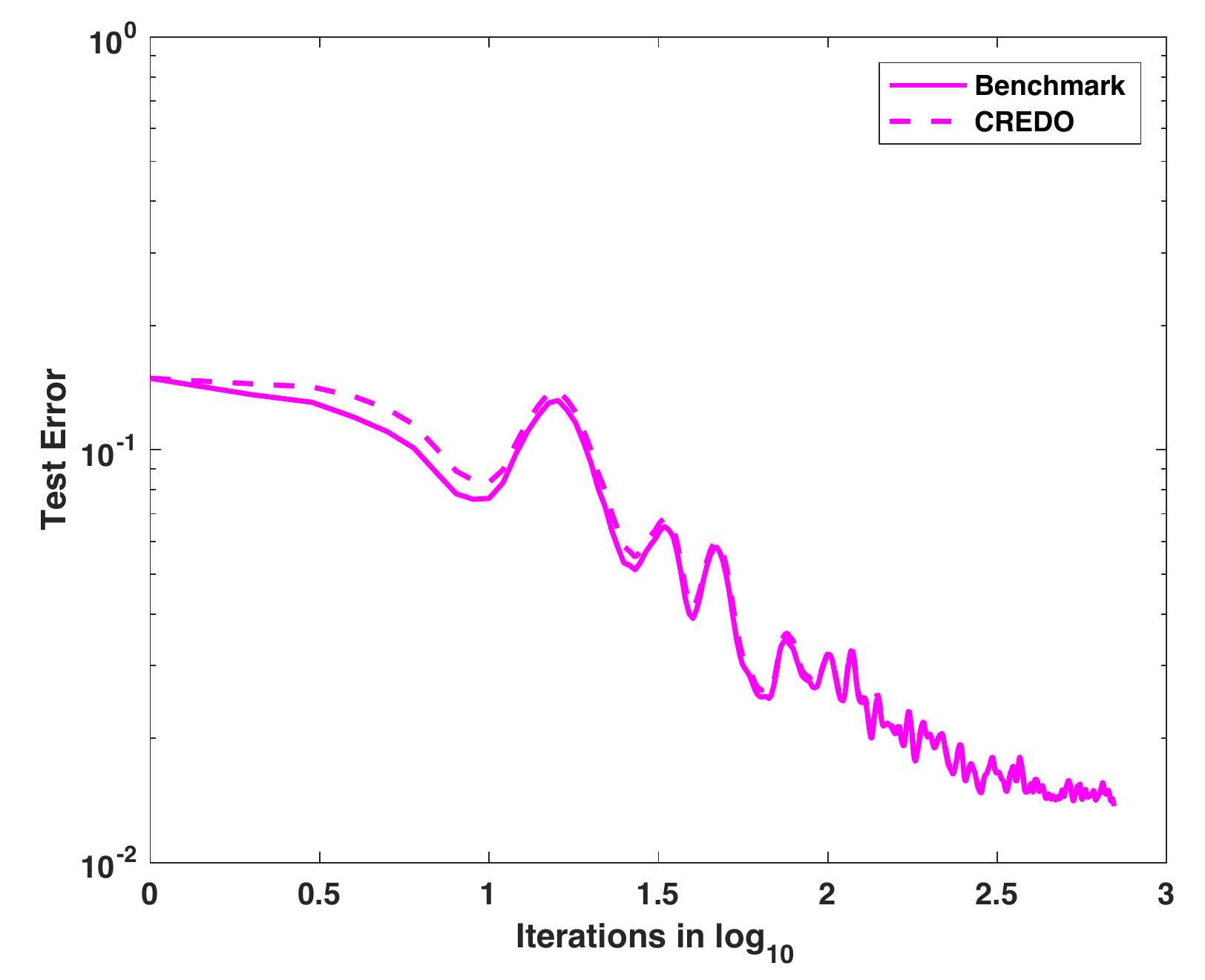}
			\caption{Comparison of Test Error: Number of Iterations}
			\label{fig:bank_mse}
		\end{subfigure}%
		\begin{subfigure}{.5\textwidth}
			\centering
			\includegraphics[width=.9\linewidth]{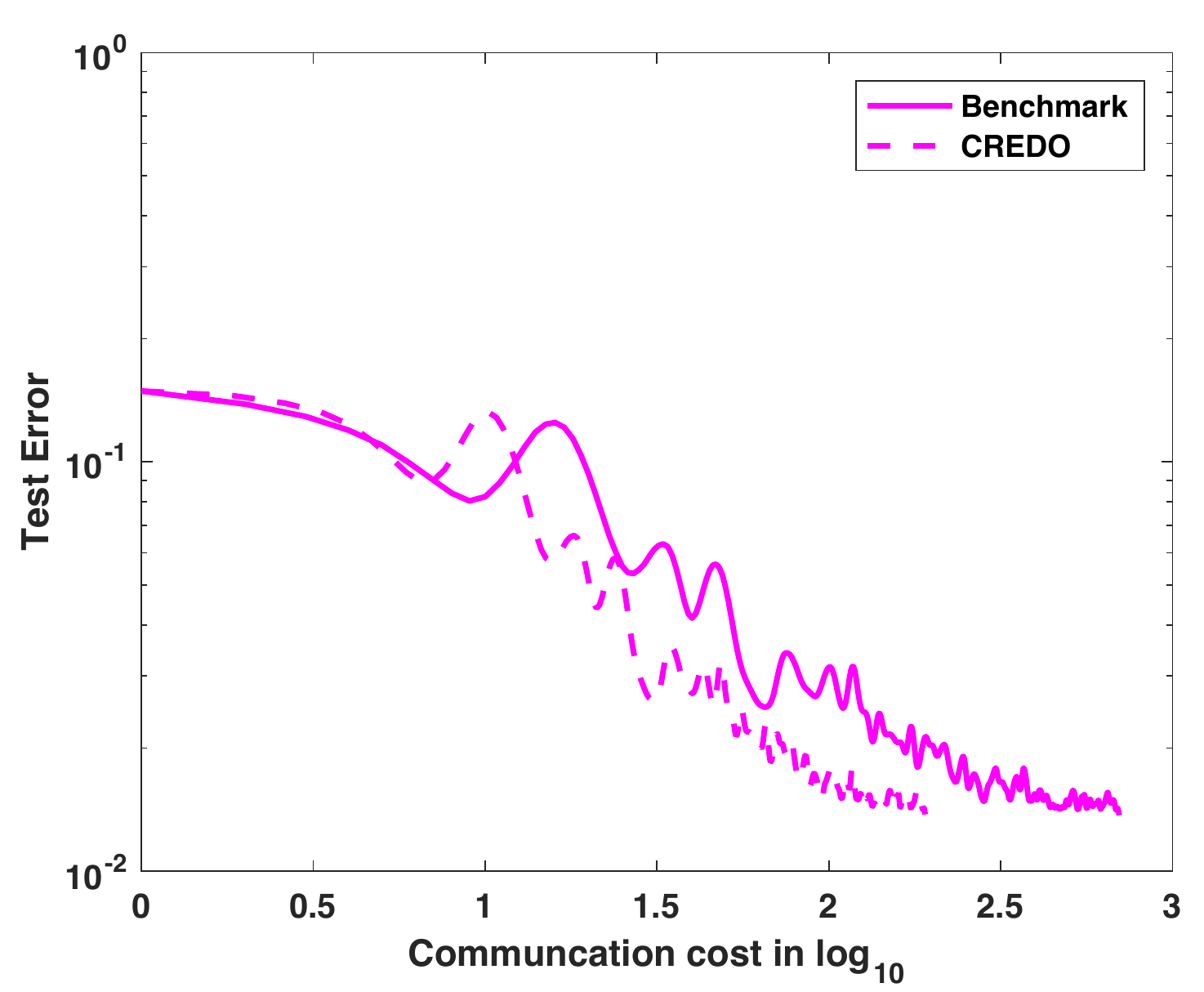}
			\caption{Comparison of Test Error: Communication cost per node}
			\label{fig:bank_comm_cost}
		\end{subfigure}
		\caption{Bank Dataset: Comparison of
			the $\mathcal{CREDO}$ and benchmark estimators}
		\label{fig:bank}
	\end{figure}
\begin{table}[h]
	\centering
	\caption{$\mathcal{CREDO}$: Communication cost across three datasets}\label{tab:2}
	\begin{tabular}{ l  l  l  l  l l}
		{\small\textbf{Dataset}} & {\small\textbf{Test Error}} & {\small\textbf{Network size}}&{\small\textbf{Avg. degree}}&{\small \textbf{$\mathcal{CREDO}$}}&{\small \textbf{Benchmark}}\\
		\hline \\
		$\textit{CADATA}$ & $2.15$&$20$&$4.8$&$\textbf{894}$ & $1810$ \\
		$\textit{ABALONE}$& $0.95$ &$10$&$5.2$&$\textbf{564}$ & $1558$ \\
		$\textit{BANK}$ & $0.015$ &$20$&$7.9$&$\textbf{1994}$ & $6962$ \\
		\hline
	\end{tabular}
\end{table}
\section{Conclusion}
\label{sec:conc}
In this paper, we have proposed a communication efficient distributed recursive estimation scheme~$\mathcal{CREDO}$, for which we have established strong consistency of the estimation sequence and characterized the asymptotic covariance of the estimate sequence in terms of the sensing model and the noise covariance. The communication efficiency of the proposed estimator has been characterized in terms of the dependence of the MSE decay on the communication cost. Specifically, we have established that the MSE decay rate of $\mathcal{CREDO}$ with respect to the number of communications can be as good as $\Theta\left(\mathcal{C}_{t}^{-2+\zeta}\right)$, where $\zeta>0$ and $\zeta$ is arbitrarily small. Future research directions include the development of communication schemes, that are adaptive in terms of the connectivity of a node, and local decision making in terms of whether to communicate or not based on neighborhood information.	The algorithm presented in this paper can be thought of as a distributed method to solve a stochastic optimization problem with a stochastic least squares-type cost function. A natural direction is to extend the proposed ideas to general stochastic distributed optimization.
\vspace{10pt}
\acks{The work of AKS and SK was supported in part by the National Science Foundation under grant CCF-1513936. The work of DJ was supported by the Ministry of Education, Science and
Technological Development, Republic of Serbia, grant no. 174030.}
\newpage
\appendix
\section*{Appendix A.}
\label{sec:conv}
We present the proofs of main results in this section.\\
\textbf{Proof sketch of Theorem \ref{th:cons}}\\
The proof of almost sure convergence of the estimate sequence to $\btheta$ involves establishing the boundedness of the estimate sequence. With the boundedness of the estimate sequence in place, we show the convergence of the estimate sequence to its averaged estimate sequence $\{\mathbf{x}_{\mbox{avg}}(t)\}$, where $\mathbf{x}_{\mbox{avg}}(t) = \frac{1}{N}\sum_{n=1}^{N}\mathbf{x}_{n}(t)$ at a rate faster $t^{1/2}$ and finally show that the averaged estimate sequence converges to $\btheta$ with a rate $\{(t+1)^{\tau}\}$~$\tau\in [0, 1/2)$. The final result follows by noting that, the averaged estimate sequence and the estimate sequence are indistinguishable in the $\{(t+1)^{\tau}\}$ time scale, where $\tau\in [0, 1/2)$.\\
\textbf{Proof sketch of Theorem \ref{th:2}}\\
The proof of the asymptotic normality of the estimate sequence proceeds in the following procedure. The first step involves establishing the asymptotic normality of the averaged estimate sequence $\mathbf{x}_{\mbox{avg}}(t)$. Moreover, an intermediate result ensures that the averaged estimate sequence and the estimate sequence are indistinguishable in the $\{(t+1)^{\frac{1}{2}}\}$ time scale. With the above development in place, it follows that the asymptotic normality of the averaged estimate sequence $\mathbf{x}_{\mbox{avg}}(t)$ can be extended to that of the estimate sequence $\{\mathbf{x}_{n}(t)\}$.\\
\begin{Lemma}
	\label{le:l0}
	For each $n$, the process $\{\mathbf{x}_{n}(t)\}$ satisfies
	\label{l2}
\begin{align}
		\label{eq:7.1}
		\mathbb{P}_{\theta}\left(\sup_{t\ge 0} \left\|\mathbf{x}(t)\right\| < \infty\right) =1.
		\end{align}
\end{Lemma}
\begin{proof}
	We first note that,
	\begin{align}
	\label{eq:la_decom}
	\mathbf{L}(t)=\beta_{t}\overline{\mathbf{L}}+\widetilde{\mathbf{L}}(t),
	\end{align}
	where $\mathbb{E}\left[\widetilde{\mathbf{L}}(t)\right] = \mathbf{0}$ and $\mathbb{E}\left[\widetilde{\mathbf{L}}_{i,j}^{2}(t)\right] = \frac{c_{4}}{(t+1)^{\tau_{1}+\epsilon}} - \frac{c_{3}^{2}}{(t+1)^{2\tau_{1}}}$.\\
	Define, $\mathbf{z}(t) = \mathbf{x}(t)-\mathbf{1}_{N}\otimes\boldsymbol{\theta}^{\ast}$ and $V(t) = \left\|\mathbf{z}(t)\right\|^{2}$.
	By conditional independence, we have that,
	\begin{align}
	\label{eq:lyap1}
	&\mathbb{E}\left[V(t+1)|\mathcal{F}_{t}\right] = V(t)\nonumber\\&+ \mathbf{z}^{\top}(t)\left(\mathbf{I}_{NM}-\beta_{t}\left(\overline{\mathbf{L}}\otimes\mathbf{I}_{M}\right)-\alpha_{t}\mathbf{G}_{H}\mathbf{\Sigma}^{-1}\mathbf{G}_{H}^{\top}\right)^{2}\mathbf{z}(t)\nonumber\\&+\mathbf{z}^{\top}(t)\mathbb{E}_{\boldsymbol{\theta}^{\ast}}\left[\left(\widetilde{\mathbf{L}}(t)\otimes\mathbf{I}_{M}\right)^{2}\right]\mathbf{z}(t)\nonumber\\
	&+\alpha^{2}(t)\mathbb{E}_{\boldsymbol{\theta}^{\ast}}\left[\left\|\mathbf{G}_{H}\mathbf{\Sigma}^{-1}\left(\mathbf{y}(t)-\mathbf{G}_{H}^{\top}\mathbf{1}_{N}\otimes\boldsymbol{\theta}^{\ast}\right)\right\|^{2}\right]\nonumber\\&-2\mathbf{z}^{\top}(t)\left(\beta_{t}\left(\overline{\mathbf{L}}\otimes\mathbf{I}_{M}\right)+\alpha_{t}\mathbf{G}_{H}\mathbf{\Sigma}^{-1}\mathbf{G}_{H}^{\top}\right)\mathbf{z}(t),
	\end{align}
	where the filtration $\{\mathcal{F}_{t}\}$ may be taken to be the natural filtration generated by the random observations, the random Laplacians~i.e.,
	\begin{align}
	\label{eq:filt_1}
	\mathcal{F}_{t}=\mathbf{\sigma}\left(\left\{\left\{\mathbf{y}_{n}(s)\right\}_{n=1}^{N}, \left\{\mathbf{L}(s)\right\}\right\}_{s=0}^{t-1}\right),
	\end{align}
	\noindent which is the $\sigma$-algebra induced by the observation processes.
	\noindent For $t\geq t_{1}$, it can be shown that,
	\begin{align}
	\label{eq:ineq1}
	&\mathbf{z}^{\top}(t)\left(\mathbf{I}_{NM}-\beta_{t}\left(\overline{\mathbf{L}}\otimes\mathbf{I}_{M}\right)-\alpha_{t}\mathbf{G}_{H}\mathbf{\Sigma}^{-1}\mathbf{G}_{H}^{\top}\right)^{2}\mathbf{z}(t) \nonumber\\&\leq \left(1-c_{4}\alpha_{t}\right)^{2}\left\|\mathbf{z}(t)\right\|^{2}.
	\end{align}
	We use the following inequalities so as to analyze the recursion in \eqref{eq:lyap1}.
	\begin{align}
	\label{eq:ineq2}
	&\mathbf{z}^{\top}(t)\mathbb{E}_{\boldsymbol{\theta}^{\ast}}\left[\left(\widetilde{\mathbf{L}}(t)\otimes\mathbf{I}_{M}\right)^{2}\right]\mathbf{z}(t) \le \frac{c_{5}\left\|\mathbf{z}_{\mathcal{C}^{\perp}}\right\|^{2}}{(t+1)^{\tau_{1}+\epsilon}} \nonumber\\
	&\mathbb{E}_{\boldsymbol{\theta}^{\ast}}\left[\left\|\mathbf{G}_{H}\mathbf{\Sigma}^{-1}\left(\mathbf{y}(t)-\mathbf{G}_{H}^{\top}\mathbf{1}_{N}\otimes\boldsymbol{\theta}^{\ast}\right)\right\|^{2}\right] \le c_{6}\nonumber\\
	&\mathbf{z}^{\top}(t)\left(\beta_{t}\left(\overline{\mathbf{L}}\otimes\mathbf{I}_{M}\right)+\alpha_{t}\mathbf{G}_{H}\mathbf{\Sigma}^{-1}\mathbf{G}_{H}^{\top}\right)\mathbf{z}(t)\nonumber\\&\geq \beta_{t}\lambda_{2}\left(\overline{\mathbf{L}}\right)\left\|\mathbf{z}_{\mathcal{C}^{\perp}}\right\|^{2}+c_{7}\alpha_{t}\left\|\mathbf{z}(t)\right\|^{2}.
	\end{align}
	Using the inequalities derived in \eqref{eq:ineq2}, we have,
	\begin{align}
	\label{eq:ineq3}
	&\mathbb{E}\left[V(t+1)|\mathcal{F}_{t}\right]\le (1+c_{8}\alpha^{2}(t))V(t)\nonumber\\&-c_{9}\left(\beta_{t}-\frac{c_{5}}{(t+1)^{\tau_{1}+\epsilon}}\right)\left\|\mathbf{z}_{\mathcal{C}^{\perp}}\right\|^{2}+c_{6}\alpha^{2}(t).
	\end{align}
	\noindent As $\frac{c_{5}}{(t+1)^{\tau_{1}+\epsilon}}$ goes to zero faster than $\beta_{t}$, $\exists t_{2}$ such that $\forall t\ge t_{2}$, $\beta_{t} \ge \frac{c_{5}}{(t+1)^{\tau_{1}+\epsilon}}$.
	\noindent By the above construction we obtain $\forall t \geq t_{2}$,
	\begin{align}
	\label{eq:l2_pr_17}
	\mathbb{E}_{\theta^{*}}[V(t+1) | \mathcal{F}_{t}] \le (1+\alpha^{2}(t))V(t)+\widehat\alpha_{t}^{2},
	\end{align}
	where $\widehat{\alpha}(t) = \sqrt{c_{6}}\alpha_{t}$.
	\noindent The product $\prod_{s=t}^{\infty}(1+\alpha_{s}^{2})$ exists for all $t$. Now let $\{W(t)\}$ be such that
	\begin{align}
	\label{eq:l2_pr_19}
	W(t)=\left(\prod_{s=t}^{\infty}(1+\alpha_{s}^{2})\right)V_{2}(t)+\sum_{s=t}^{\infty}\widehat{\alpha}_{s}^{2},~\forall t\geq t_{2}.
	\end{align}
	\noindent By \eqref{eq:l2_pr_19}, it can be shown that $\{W(t)\}$ satisfies,
	\begin{align}
	\label{eq:l2_pr_20}
	\mathbb{E}_{\theta^{*}}[W(t+1) | \mathcal{F}_{t}] \le W(t).
	\end{align}
	\noindent Hence,  $\{W(t)\}$ is a non-negative super martingale and converges a.s. to a bounded random variable $W^{*}$ as $t\to\infty$. It then follows from \eqref{eq:l2_pr_19} that $V(t)\to W^{*}$ as $t\to\infty$. Thus, we conclude that the sequences $\{\mathbf{x}_{n}(t)\}$ are bounded for all $n$.
\end{proof}
We now prove the almost sure convergence of the estimate sequence to the true parameter. In the sequel, we establish the order optimal convergence of the estimate sequence in the regime of $0<\tau_{1}<\frac{1}{2}-\frac{1}{2+\epsilon_{1}}$.
\begin{Lemma}
	\label{le:conv}
	Let the hypothesis of Theorem \ref{th:cons} hold. Then, we have,
	\begin{align}
	\label{eq:conv}
	\mathbb{P}_{\btheta}\left(\lim_{t\rightarrow\infty}\mathbf{x}_{n}(t)=\btheta\right)=1.
	\end{align}
	\begin{proof}[Proof of Lemma \ref{le:conv}]
		Following as in the proof of Lemma \ref{le:l0}, for $t$ large enough
		\begin{align}
		\label{eq:l1_pr_13}
		&\mathbb{E}_{\btheta}[V(t+1)|\mathcal{F}_{t}]\le\left(1-2c_{4}\alpha_{t}+c_{7}\alpha^{2}_{t}\right)V(t)+c_{6}\alpha_{t}^{2}\nonumber\\
		&\le V(t)+c_{6}\alpha_{t}^{2},
		\end{align}
		\noindent as for $t$ large enough, $-2c_{4}\alpha_{t}+c_{7}\alpha^{2}_{t}<0$.
		\noindent Now, consider the $\{\mathcal{F}_{t}\}$-adapted process $\{V_{1}(t)\}$ defined as follows
		\begin{align}
		\label{eq:l1_pr_14}
		&V_{1}(t)=V(t)+c_{6}\sum_{s=t}^{\infty}\alpha_{s}^{2}\nonumber\\
		&=V(t)+c_{8}\sum_{s=t}^{\infty}(t+1)^{-2},
		\end{align}		
		\noindent for appropriately chosen positive constant $c_{8}$.Since, $\{(t+1)^{-2}\}$ is summable, the process $\{V_{1}(t)\}$ is bounded from above. Moreover, it also follows that $\{V_{1}(t)\}_{t\geq t_{1}}$ is a supermartingale and hence converges a.s. to a finite random variable. By definition from \eqref{eq:l1_pr_14}, we also have that $\{V(t)\}$ converges to a non-negative finite random variable $V^{*}$. Finally, from \eqref{eq:l1_pr_13}, we have that,
		\begin{align}
		\label{eq:l1_pr_15}
		\mathbb{E}_{\btheta}[V(t+1)]\le \left(1-c_{7}\alpha_{t}\right)\mathbb{E}_{\btheta}[V(t)]+c_{9}(t+1)^{-2},
		\end{align}
		\noindent for $t\geq t_{1}$. The sequence $\{V(t)\}$ then falls under the purview of Lemma \ref{int_res_0}, and we have $\mathbb{E}_{\btheta}[V(t)]\to 0$ as $t\to\infty$. Finally, by Fatou's Lemma, where we use the non-negativity of the sequence $\{V(t)\}$, we conclude that
		\begin{align}
		\label{eq:l1_pr_16}
		0\leq \mathbb{E}_{\btheta}[V^{*}]\le\liminf_{t\to\infty}\mathbb{E}_{\btheta}[V(t)]=0,
		\end{align}
		\noindent which thus implies that $V^{*}=0$ a.s. Hence, $\left\|\mathbf{z}(t)\right\|\to 0$ as $t\to\infty$ and the desired assertion follows.
	\end{proof}
\end{Lemma}
\noindent Consider the averaged estimate sequence, $\{\mathbf{x}_{\mbox{avg}}(t)\}\}$, which follows the following update:
\begin{align}
\label{eq:xavg}
&\mathbf{x}_{\mbox{avg}}(t+1)=\left(\mathbf{I}_{M}-\frac{\alpha_{t}}{N}\sum_{n=1}^{N}\mathbf{H}_{n}^{\top}\mathbf{\Sigma}_{n}^{-1}\mathbf{H}_{n}\right)\mathbf{x}_{\mbox{avg}}(t)\nonumber\\
&+\frac{\alpha_{t}}{N}\sum_{n=1}^{N}\mathbf{H}_{n}^{\top}\mathbf{\Sigma}_{n}^{-1}\left(\mathbf{x}_{n}(t)-\mathbf{x}_{\mbox{avg}}(t)\right)\nonumber\\&+\frac{\alpha_{t}}{N}\sum_{n=1}^{N}\mathbf{H}_{n}^{\top}\mathbf{\Sigma}_{n}^{-1}\mathbf{\gamma}_{n}(t).
\end{align}
\noindent The following Lemmas will be used to quantify the rate of convergence of distributed vector or matrix valued recursions to their network-averaged behavior.
\noindent\begin{Lemma}
	\label{le:l1.1}
	Let $\{z_{t}\}$ be an $\mathbb{R}^{+}$ valued $\mathcal{F}_{t}$-adapted process that satisfies
	\begin{align*}
	z_{t+1} \leq \left(1-r_{1}(t)\right)z_{t} +r_{2}(t)U_t(1+J_t),
	\end{align*}
	\noindent where $\{r_1(t)\}$ is an $\mathcal{F}_{t+1}$-adapted process, such that for all $t$, $r_1(t)$ satisfies $0\leq r_1(t)\leq 1$ and
	\begin{align*}
	a_{1}\leq \mathbb{E}\left[r_1(t)|\mathcal{F}_{t}\right] \leq \frac{1}{(t+1)^{\delta_{1}}}
	\end{align*}
	\noindent with $a_{1} > 0$ and $0\leq \delta_{1} < 1$. The sequence $\{r_2(t)\}$ is deterministic and $\mathbb{R}^{+}$ valued and satisfies $r_2(t)\leq \frac{a_{2}}{(t+1)^{\delta_{2}}}$ with $a_2 > 0$ and $\delta_{2}>0$. Further, let $\{U_t\}$ and $\{J_t\}$ be $\mathbb{R}^{+}$ valued $\mathcal{F}_{t}$ and $\mathcal{F}_{t+1}$ adapted processes, respectively, with $\sup_{t\geq 0} \left\|U_t\right\|<\infty$ a.s. The process $\left\{J_{t}\right\}$ is i.i.d. with $J_{t}$ independent of $\mathcal{F}_{t}$ for each $t$ and satisfies the moment condition $\mathbb{E}\left[\left\|J_{t}\right\|^{2+\epsilon_{1}} \right]<\kappa<\infty$ for some $\epsilon_{1}>0$ and a constant $\kappa > 0$. Then, for every $\delta_{0}$ such that
	$0\leq \delta_0 < \delta_{2}-\delta_{1}-\frac{1}{2+\epsilon_{1}}$, we have $(t+1)^{\delta_{0}}z_{t}\to 0$ a.s. as $t\to\infty$.
\end{Lemma}
\begin{Lemma}[Lemma 4.1 in \citet{kar2013distributed}]
	\label{int_res_0}
	\noindent Consider the scalar time-varying linear system
	\begin{align}
	\label{eq:int_res_0}
	u(t+1)\leq(1-r_{1}(t))u(t)+r_{2}(t),
	\end{align}
	\noindent where $\{r_{1}(t)\}$ is a sequence, such that
	\begin{align}
	\label{eq:int_res_0_1}
	\frac{a_{1}}{(t+1)^{\delta_{1}}}\leq r_{1}(t)\leq 1
	\end{align}
	\noindent with $a_{1} >0, 0\leq\delta_{1}\leq 1$, whereas the sequence $\{r_{2}(t)\}$ is given by
	\begin{align}
	\label{eq:int_res_0_2}
	r_{2}(t)\le\frac{a_{2}}{(t+1)^{\delta_{2}}}
	\end{align}
	\noindent with $a_{2}>0, \delta_{2}\geq 0$. Then, if $u(0)\geq 0$ and $\delta_{1} < \delta_{2}$, we have
	\begin{align}
	\label{eq:int_res_0_3}
	\lim_{t\to\infty}(t+1)^{\delta_0}u(t)=0,
	\end{align}
	\noindent for all $0\le\delta_{0}<\delta_{2}-\delta_{1}$. Also, if $\delta_{1}=\delta_{2}$, then the sequence $\{u(t)\}$ stays bounded, i.e. $\sup_{t\geq 0}\left\|u(t)\right\|<\infty$.
\end{Lemma}
\begin{Lemma}
	\label{le:l2}
	Let the Assumptions \ref{m:1}-\ref{m:3} hold. Consider the averaged estimate sequence as in \eqref{eq:xavg}. Then, we have,
	\begin{align}
	\label{eq:l2_1}
	\mathbb{P}\left(\lim_{t\to\infty}(t+1)^{\frac{1}{2}+\delta}\left(\mathbf{x}(t)-\mathbf{1}_{N}\otimes\mathbf{x}_{avg}(t)\right)=0\right)=1
	\end{align}
\end{Lemma}
\begin{proof}
	Let $\mathcal{L}_{t}$ denote the set of possible Laplacian matrices~(necessarily finite) at time $t$. Note, that the finiteness property of the cardinality of the set $\mathcal{L}_{t}$ holds for all $t$. Since the set of Laplacians is finite, we have,
	\begin{align}
	\label{eq:l2_2}
	\underline{p}=\inf_{\mathbf{L}\in\mathcal{L}_{t}}p_{\mathbf{L}} > 0,
	\end{align}
	with $p_{L}=\mathbb{P}\left(\mathbf{L}(t)=\mathbf{L}\right)$ for each $\mathbf{L}\in\mathcal{L}_{t}$ such that $\sum_{\mathbf{L}\in\mathcal{L}_{t}}p_{\mathbf{L}}=1$.
	Assumption \ref{m:3}, i.e., $\lambda_{2}\left(\overline{\mathbf{L}}(t)\right) > 0$ implies that for every $\mathbf{z}\in\mathcal{C}^{\perp}$, where,
	\begin{align}
	\label{eq:l2_3}
	\mathcal{C} = \left\{\mathbf{x}|\mathbf{x}=\mathbf{1}_{N}\otimes \mathbf{a}, \mathbf{a}\in\mathbb{R}^{M}\right\},
	\end{align}
	we have,
	\begin{align}
	\label{eq:l2_4}
	\sum_{\mathbf{L}\in\mathcal{L}_{t}}\mathbf{z}^{\top}\mathbf{L}\mathbf{z}\geq \sum_{\mathbf{L}\in\mathcal{L}_{t}}\mathbf{z}^{\top}p_{\mathbf{L}}\mathbf{L}\mathbf{z}=\mathbf{z}^{\top}\overline{\mathbf{L}}(t)\mathbf{z}\geq \lambda_{2}\left(\overline{\mathbf{L}}(t)\right)\left\|\mathbf{z}\right\|^{2}.
	\end{align}
	Owing to the finite cardinality of $\mathcal{L}_{t}$ and \eqref{eq:l2_4}, we also have that for each $\mathbf{z}\in\mathcal{C}^{\perp}$,$\exists \mathbf{L}_{\mathbf{z}}\in\mathcal{L}_{t}$ such that,
	\begin{align}
	\label{eq:l2_41}
	\mathbf{z}^{\top}\mathbf{L}_{\mathbf{z}}\mathbf{z}\ge \frac{\lambda_{2}\left(\overline{\mathbf{L}}(t)\right)}{|\mathcal{L}_{t}|}\left\|\mathbf{z}\right\|^{2}
	\end{align}
	Moreover, since $\mathcal{L}_{t}$ is finite, the mapping $L_{\mathbf{z}}:\mathcal{C}^{\perp}\mapsto\mathcal{L}_{t}$ can be realized as a measurable function. It is also to be noted that, $\mathbf{L}(t)=\rho_{t}^{2}\widehat{\mathbf{L}}$, where $\widehat{\mathbf{L}}$ is a Laplacian such that $[\widehat{\mathbf{L}}]_{ij}\in\mathbb{Z}$. For each, $\mathbf{L}\in\mathcal{L}_{t}$, the eigen values of $\mathbf{I}_{NM}-\rho_{t}^{2}\left(\widehat{\mathbf{L}}\otimes\mathbf{I}_{M}\right)$ are given by $M$ repetitions of$1$ and $1-\rho_{t}^{2}\lambda_{n}\left(\widehat{\mathbf{L}}\right)$, where $2\leq n\leq N$.
	Thus, for $t\geq t_{0}$, $\left\|\mathbf{I}_{NM}-\rho_{t}^{2}\left(\widehat{\mathbf{L}}\otimes\mathbf{I}_{M}\right)\right\| \leq 1$ and $\left\|\left(\mathbf{I}_{NM}-\rho_{t}^{2}\left(\widehat{\mathbf{L}}\otimes\mathbf{I}_{M}\right)\right)\mathbf{z}\right\| \leq \left\|\mathbf{z}\right\|$. Hence, we can define a jointly measurable function $r_{\mathbf{L},\mathbf{z}}$ given by,
	\begin{align}
	\label{l2_5}
	r_{\mathbf{L},\mathbf{z}} =
	\begin{cases}
	1 &~~\textit{if}~t<t_{0}~\textit{or}~\mathbf{z}=\mathbf{0}\\
	1-\frac{\left\|\left(\mathbf{I}_{NM}-\rho_{t}^{2}\left(\widehat{\mathbf{L}}\otimes\mathbf{I}_{M}\right)\right)\mathbf{z}\right\|}{\left\|\mathbf{z}\right\|} & ~~\textit{otherwise},
	\end{cases}
	\end{align}
	which satisfies $0\le r_{\mathbf{L},\mathbf{z}} \le 1$ for each $\left(\mathbf{L},\mathbf{z}\right)$.
	\noindent Define $\{r_{t}\}$ to be a $\mathcal{F}_{t+1}$ process given by, $r_{t} = r_{\mathbf{L}(t),\mathbf{z}_{t}}$ for each $t$ and $\left\|\left(\mathbf{I}_{NM}-\rho_{t}^{2}\left(\widehat{\mathbf{L}}\otimes\mathbf{I}_{M}\right)\right)\mathbf{z}_{t}\right\|=(1-r_{t})\left\|\mathbf{z}_{t}\right\|$ a.s. for each $t$.
	\noindent Then, we have,
	\begin{align}
	\label{eq:l2_6}
	&\left\|\left(\mathbf{I}_{NM}-\rho_{t}^{2}\left(\widehat{\mathbf{L}}_{\mathbf{z}_{t}}\otimes\mathbf{I}_{M}\right)\right)\mathbf{z}_{t}\right\|^{2}\nonumber\\&=\mathbf{z}^{\top}_{t}\left(\mathbf{I}_{NM}-2\rho_{t}^{2}\left(\widehat{\mathbf{L}}\otimes\mathbf{I}_{M}\right)\right)\mathbf{z}_{t}\nonumber\\
	&+\mathbf{z}_{t}^{\top}\rho_{t}^{4}\left(\widehat{\mathbf{L}}_{\mathbf{z}_{t}}\otimes\mathbf{I}_{M}\right)^{2}\mathbf{z}_{t}\nonumber\\
	&\leq \left(1-2\beta_{t}\frac{\lambda_{2}\left(\overline{\mathbf{L}}\right)}{|\mathcal{L}_{t}|}\right)\left\|\mathbf{z}_{t}\right\|^{2}+c_{1}\rho_{t}^{4}\left\|\mathbf{z}_{t}\right\|^{2}\nonumber\\
	&\leq \left(1-\beta_{t}\frac{\lambda_{2}\left(\overline{\mathbf{L}}\right)}{|\mathcal{L}_{t}|}\right)\left\|\mathbf{z}_{t}\right\|^{2}
	\end{align}
	where we have used the boundedness of the Laplacian matrix and the fact that $\overline{\mathbf{L}}_{t}=\beta_{t}\overline{\mathbf{L}}$.
	With the above development in place, choosing an appropriate $t_{1}$~(making $t_{0}$ larger if necessary), for all $t\geq t_{1}$, we have,
	\begin{align}
	\label{eq:l2_7}
	\left\|\left(\mathbf{I}_{NM}-\rho_{t}^{2}\left(\widehat{\mathbf{L}}_{\mathbf{z}_{t}}\otimes\mathbf{I}_{M}\right)\right)\mathbf{z}_{t}\right\|\le \left(1-\beta_{t}\frac{\lambda_{2}\left(\overline{\mathbf{L}}\right)}{4|\mathcal{L}_{t}|}\right)\left\|\mathbf{z}_{t}\right\|^{2}.
	\end{align}
	Then, from \eqref{eq:l2_7}, we have,
	\begin{align}
	\label{eq:l2_8}
	&\mathbb{E}\left.\left[\left\|\left(\mathbf{I}_{NM}-\rho_{t}^{2}\left(\widehat{\mathbf{L}}_{\mathbf{z}_{t}}\otimes\mathbf{I}_{M}\right)\right)\mathbf{z}_{t}\right\|\right|\mathcal{F}_{t}\right]\nonumber\\&=\sum_{\mathbf{L}\in\mathcal{L}_{t}}p_{\mathbf{L}}\left(1-r_{\mathbf{L},\mathbf{z}_{t}}\right)\left\|\mathbf{z}_{t}\right\|\nonumber\\
	&\le \left(1-\left(\underline{p}\beta_{t}\frac{\lambda_{2}\left(\overline{\mathbf{L}}\right)}{4|\mathcal{L}_{t}|}+\sum_{\mathbf{L}\neq\mathbf{L}_{\mathbf{z}_{t}}}\right)\right)\left\|\mathbf{z}_{t}\right\|.
	\end{align}
	Since, $\sum_{\mathbf{L}\neq\mathbf{L}_{\mathbf{z}_{t}}} p_{\mathbf{L}}r_{\mathbf{L},\mathbf{z}_{t}}\geq 0$, we have for all $t\ge t_1$,
	\begin{align}
	\label{eq:l2_9}
	&\left(1-\mathbb{E}\left[r_{t}|\mathcal{F}_{t}\right]\right)\left\|\mathbf{z}_{t}\right\| \nonumber\\&= \mathbb{E}\left.\left[\left\|\left(\mathbf{I}_{NM}-\rho_{t}^{2}\left(\widehat{\mathbf{L}}_{\mathbf{z}_{t}}\otimes\mathbf{I}_{M}\right)\right)\mathbf{z}_{t}\right\|\right|\mathcal{F}_{t}\right]\nonumber\\
	&\le \left(1-\underline{p}\beta_{t}\frac{\lambda_{2}\left(\overline{\mathbf{L}}\right)}{4|\mathcal{L}_{t}|}\right)\left\|\mathbf{z}_{t}\right\|.
	\end{align}
	\noindent As $r_{t}=1$ on the set $\{\mathbf{z}_{t}=0\}$, we have that,
	\begin{align}
	\label{eq:l2_10}
	\mathbb{E}\left[r_{t}|\mathcal{F}_{t}\right]\geq \underline{p}\beta_{t}\frac{\lambda_{2}\left(\overline{\mathbf{L}}\right)}{4|\mathcal{L}_{t}|}.
	\end{align}
	Thus, we have established that,
	\begin{align}
	\label{eq:l2_11}
	\left\|\left(\mathbf{I}_{NM}-\left(\mathbf{L}(t)\otimes\mathbf{I}_{M}\right)\right)\mathbf{z}_{t}\right\|\le (1-r_{t})\left\|\mathbf{z}_{t}\right\|,
	\end{align}
	where $\{r_t\}$ is a $\mathbb{R}^{+}$ valued $\mathcal{F}_{t+1}$ process satisfying \eqref{eq:l2_10}.
	With the above development in place, consider the residual process $\{\widetilde{\mathbf{x}}(t)\}$ given by $\widetilde{\mathbf{x}}(t) = \mathbf{x}(t)-\mathbf{x}_{\mbox{avg}}(t)$. Thus, we have that the process $\{\widetilde{\mathbf{x}}(t)\}$ satisfies the recursion,
	\begin{align}
	\label{eq:l2_12}
	\widetilde{\mathbf{x}}(t+1) = \left(\mathbf{I}_{NM}-\mathbf{L}(t)\otimes\mathbf{I}_{M}\right)\widetilde{\mathbf{x}}(t)+\alpha_{t}\widetilde{\mathbf{z}}(t),
	\end{align}
	where the process $\{\widetilde{\mathbf{z}}(t)\}$ is given by
	\begin{align}
	\label{eq:l2_13}
	&\widetilde{\mathbf{z}}(t)=\left(\mathbf{I}_{NM}-\frac{1}{N}\mathbf{1}_{N}\otimes\left(\mathbf{1}_{N}\otimes\mathbf{I}_{M}\right)^{\top}\right)\nonumber\\&\times\mathbf{G}_{H}\mathbf{\Sigma}^{-1}\left(\mathbf{y}(t)-\mathbf{G}_{H}^{\top}\mathbf{x}(t)\right).
	\end{align}
	\noindent From \eqref{eq:l2_13}, we also have,
	\begin{align}
	\label{eq:l2_14}
	\widetilde{\mathbf{z}}(t)=\overline{\mathbf{J}}_{t}+\overline{\mathbf{U}}_{t},
	\end{align}
	where,
	\begin{align}
	\label{eq:l2_141}
	&\overline{\mathbf{J}}_{t}=\left(\mathbf{I}_{NM}-\frac{1}{N}\mathbf{1}_{N}\otimes\left(\mathbf{1}_{N}\otimes\mathbf{I}_{M}\right)^{\top}\right)\nonumber\\&\times\mathbf{G}_{H}\mathbf{\Sigma}^{-1}\left(\mathbf{y}(t)-\mathbf{G}_{H}^{\top}\btheta\right)\nonumber\\
	&\overline{\mathbf{U}}_{t}=\left(\mathbf{I}_{NM}-\frac{1}{N}\mathbf{1}_{N}\otimes\left(\mathbf{1}_{N}\otimes\mathbf{I}_{M}\right)^{\top}\right)\nonumber\\&\times\mathbf{G}_{H}\mathbf{\Sigma}^{-1}\left(\mathbf{G}_{H}^{\top}\btheta-\mathbf{G}_{H}^{\top}\mathbf{x}(t)\right).
	\end{align}
	By Lemma \ref{le:l0}, we also have that, the process $\{\mathbf{x}(t)\}$ is bounded. Hence, there exists an $\mathcal{F}_{t}$-adapted process $\{\widetilde{U}_{t}\}$ such that $\left\|\overline{\mathbf{U}}_{t}\right\|\le \widetilde{U}_{t}$ and $\sup_{t\ge 0}\widetilde{U}_{t} < \infty$ a.s.. Furthermore, denote the process $U_{t}$ as follows,
	\begin{align}
	\label{eq:l2_15}
	U_{t} = \max\left\{\widetilde{U}_{t}, \left\|\mathbf{I}_{NM}-\frac{1}{N}\mathbf{1}_{N}\otimes\left(\mathbf{1}_{N}\otimes\mathbf{I}_{M}\right)^{\top}\right\|\right\}.
	\end{align}
	With the above development in place, we conclude,
	\begin{align}
	\label{eq:l2_15.1}
	\left\|\overline{\mathbf{U}}_{t}\right\|+\left\|\overline{\mathbf{J}}_{t}\right\| \le U_{t}\left(1+J_{t}\right),
	\end{align}
	where $J_{t} = \mathbf{y}(t)-\mathbf{G}_{H}^{\top}\btheta$.
	Then, from \eqref{eq:l2_11}-\eqref{eq:l2_12} and noting that $\widetilde{\mathbf{x}}(t)\in\mathcal{C}^{\perp}$, we have,
	\begin{align}
	\label{eq:l2_16}
	\left\|\widetilde{\mathbf{x}}(t+1)\right\| \leq (1-r_{t})\left\|\widetilde{\mathbf{x}}(t)\right\|+\alpha_{t}U_{t}(1+J_{t}),
	\end{align}
	which then falls under the purview of Lemma \ref{le:l1.1} and hence we have the assertion,
	\begin{align}
	\label{eq:l2_17}
	\mathbb{P}\left(\lim_{t\to\infty}(t+1)^{\delta_0}\left(\mathbf{x}(t)-\mathbf{1}_{N}\otimes\mathbf{x}_{avg}(t)\right)=0\right)=1,
	\end{align}
	where $0<\delta_{0}<1-\tau_{1}$ and hence $\delta_{0}$ can be chosen to be $1/2 + \delta$, where $\delta > 0$ and we finally have,
	\begin{align}
	\label{eq:l2_18}
	\mathbb{P}\left(\lim_{t\to\infty}(t+1)^{\frac{1}{2}+\delta}\left(\mathbf{x}(t)-\mathbf{1}_{N}\otimes\mathbf{x}_{avg}(t)\right)=0\right)=1.
	\end{align}
\end{proof}
\begin{Lemma}
	\label{le:l1}
	Let the Assumptions \ref{m:1}-\ref{m:3} hold. Consider the averaged estimate sequence as in \eqref{eq:xavg}. Then, we have,
	\begin{align}
	\label{eq:l1_1}
	\mathbb{P}_{\btheta}\left(\lim_{t\to\infty}\mathbf{x}_{\mbox{avg}}(t)=\btheta\right)=1.
	\end{align}
\end{Lemma}
\begin{proof}
	\noindent Define, the residual sequence, $\{\mathbf{z}_{t}\}$, where $\mathbf{z}(t)=\mathbf{x}_{\mbox{avg}}(t)-\btheta$, which can be then shown to satisfy the recursion
	\begin{align}
	\label{eq:l1_2}
	\mathbf{z}_{t+1}=\left(\mathbf{I}_{M}-\alpha_{t}\Gamma\right)\mathbf{z}_{t}+\alpha_{t}\mathbf{U}_{t}+\alpha_{t}\mathbf{J}_{t},
	\end{align}
	where
	\begin{align}
	\label{eq:l1_3}
	&\Gamma = \frac{1}{N}\sum_{n=1}^{N}\mathbf{H}_{n}^{\top}\mathbf{\Sigma}_{n}^{-1}\mathbf{H}_{n}\nonumber\\
	&\mathbf{U}_{t} = \frac{1}{N}\sum_{n=1}^{N}\mathbf{H}_{n}^{\top}\mathbf{\Sigma}_{n}^{-1}\left(\mathbf{x}_{n}(t)-\mathbf{x}_{\mbox{avg}}(t)\right)\nonumber\\
	&\mathbf{J}_{t}=\frac{1}{N}\sum_{n=1}^{N}\mathbf{H}_{n}^{\top}\mathbf{\Sigma}_{n}^{-1}\mathbf{\gamma}_{n}(t).
	\end{align}	
	From, Lemma \ref{le:l2}, we have that,
	\begin{align}
	\label{eq:l1_4}
	\mathbb{P}\left(\lim_{t\to\infty}(t+1)^{\delta_0}\left(\mathbf{x}(t)-\mathbf{1}_{N}\otimes\mathbf{x}_{avg}(t)\right)=0\right)=1,
	\end{align}
	where $0<\delta_{0}<1-\tau_{1}$. Fix, a $\delta_{0}$ and then by convergence of $(t+1)^{\delta_0}\mathbf{U}_{t}\to 0$~a.s. as $t\to\infty$ and Egorov's theorem, the a.s. convergence may be assumed to be uniform on sets of arbitrarily large probability measure and hence for every $\delta > 0$, there exists uniformly bounded process $\{\mathbf{U}_{t}^{\delta}\}$ satisfying,
	\begin{align}
	\label{eq:l1_5}
	\mathbb{P}_{\btheta}\left(\sup_{s\geq t_{\epsilon}^{\delta}}(s+1)^{\delta_{0}}\left\|\mathbf{U}_{t}^{\delta}\right\|>\epsilon\right) = 0,
	\end{align}
	for each $\epsilon > 0$ and some $t_{\epsilon}^{\delta}$ chosen appropriately large enough such that
	\begin{align}
	\label{eq:l1_6}
	\mathbb{P}_{\btheta}\left(\sup_{t\geq 0}\left\|\mathbf{U}_{t}^{\delta}-\mathbf{U}_{t}\right\|=0\right)>1-\delta.
	\end{align}
	With the above development in the place, for each $\delta > 0$, define the $\mathcal{F}_{t}$-adapted process $\{\mathbf{z}_{t}^{\delta}\}$ which satisfies the recursion
	\begin{align}
	\label{eq:l1_7}
	\mathbf{z}^{\delta}_{t+1}=\left(\mathbf{I}_{M}-\alpha_{t}\Gamma\right)\mathbf{z}^{\delta}_{t}+\alpha_{t}\mathbf{U}^{\delta}_{t}+\alpha_{t}\mathbf{J}_{t}, \mathbf{z}^{\delta}_{0}=\mathbf{z}_{0},
	\end{align}
	and
	\begin{align}
	\label{eq:l1_8}
	\mathbb{P}_{\btheta}\left(\sup_{t\geq 0}\left\|\mathbf{z}^{\delta}_{t}-\mathbf{z}_{t}\right\|=0\right)>1-\delta.
	\end{align}
	\noindent It is to be noted that, in order to show that $\mathbf{z}_{t}\to 0$ as $t\to\infty$, it suffices to show that $\mathbf{z}_{t}^{\delta}\to 0$ for each $\delta > 0$. We now focus on the process $\{\mathbf{z}_{t}^{\delta}\}$ for a fixed but arbitrary $\delta > 0$.
	\noindent Let $\{V_{t}^{\delta}\}$ denote the $\mathcal{F}_{t}$-adapted process such that $V_{t}^{\delta}=\left\|\mathbf{z}_{t}^{\delta}\right\|^{2}$. Then, we have,
	\begin{align}
	\label{eq:l1_9}
	&\mathbb{E}_{\btheta}\left[V_{t+1}^{\delta}\right] \le \left\|\mathbf{I}_{M}-\alpha_{t}\mathbf{\Gamma}\right\|^{2}V_{t}^{\delta}+2\alpha_{t}\left(\mathbf{U}_{t}^{\delta}\right)^{\top}\left(\mathbf{I}_{M}-\alpha_{t}\mathbf{\Gamma}\right)\mathbf{z}_{t}^{\delta}\nonumber\\
	&+\alpha_{t}^{2}\left\|\mathbf{U}_{t}^{\delta}\right\|^{2}+\alpha_{t}^{2}\left.\mathbb{E}_{\btheta}\left[\left\|\mathbf{J}_{t}\right\|^{2}\right|\mathcal{F}_{t}\right].
	\end{align}
	\noindent For large enough $t$, we have,
	\begin{align}
	\label{eq:l1_10}
	&\left\|2\alpha_{t}\left(\mathbf{U}_{t}^{\delta}\right)^{\top}\left(\mathbf{I}_{M}-\alpha_{t}\mathbf{\Gamma}\right)\mathbf{z}_{t}^{\delta}\right\|\le 2\alpha_{t}\left\|\mathbf{U}_{t}^{\delta}\right\|\left\|\mathbf{z}_{t}^{\delta}\right\|\nonumber\\&\le 2\alpha_{t}\left\|\mathbf{U}_{t}^{\delta}\right\|\left\|\mathbf{z}_{t}^{\delta}\right\|^{2}+2\alpha_{t}\left\|\mathbf{U}_{t}^{\delta}\right\|.
	\end{align}
	We note that $\left.\mathbb{E}_{\btheta}\left[\left\|\mathbf{J}_{t}\right\|^{2}\right|\mathcal{F}_{t}\right]$ is bounded and making $t_{\epsilon}^{\delta}$ larger if necessary in order to ensure $\left\|\mathbf{U}_{t}^{\delta}\right\|\le \epsilon (t+1)^{-\delta_{0}}$, it follows that $\exists c_{1}, c_2$ such that
	\begin{align}
	\label{eq:l1_11}
	&\mathbb{E}_{\btheta}\left[V_{t+1}^{\delta}\right] \le \left(1-c_{1}\alpha_{t}+c_{2}\alpha_{t}(t+1)^{-\delta_{0}}\right)V_{t}^{\delta}\nonumber\\&+c_{2}\left(\alpha_{t}(t+1)^{-\delta_{0}}+\alpha_{t}^{2}(t+1)^{-2\delta_{0}}+\alpha_{t}^{2}\right)\nonumber\\
	&\le \left(1-c_{3}\alpha_{t}\right)V_{t}^{\delta}+c_{4}\alpha_{t}(t+1)^{-\delta_{0}}\le V_{t}^{\delta}+c_{4}\alpha_{t}(t+1)^{-\delta_{0}}
	\end{align}
	which is ensured by making $c_{4}>c_{2}$ and $c_{3} < c_{1}$ respectively. As the process $\{\alpha_{t}(t+1)^{-\delta_{0}}\}$ is summable, the process $\{\overline{V}_{t}^{\delta}\}$ given by,
	\begin{align}
	\label{eq:l1_12}
	\overline{V}_{t}^{\delta} = V_{t}^{\delta}+c_{4}\sum_{s=t}^{\infty}\alpha_{s}(s+1)^{-\delta_{0}},
	\end{align}
	is bounded from above. Thus, we have that $\{\overline{V}_{t}^{\delta}\}_{t\geq t_{\epsilon}^{\delta}}$ is a supermartingale and hence converges to a finite random variable. From \eqref{eq:l1_12}, we have that the process $\{V_{t}^{\delta}\}$ converges to a finite random variable $V^{\delta}$. We also have from \eqref{eq:l1_11}, for $t\geq t_{\epsilon}^{\delta}$
	\begin{align}
	\label{eq:l1_13}
	\mathbb{E}_{\btheta}\left[V_{t+1}^{\delta}\right] \le \left(1-c_{3}\alpha_{t}\right)\mathbb{E}_{\btheta}\left[V_{t}^{\delta}\right]+c_{4}\alpha_{t}(t+1)^{-\delta_{0}}..
	\end{align}
	\noindent Since $\delta_0 > 0$, the recursion in \eqref{eq:l1_13} falls under the purview of Lemma \ref{int_res_0} and thus we have, $\mathbb{E}_{\btheta}\left[V_{t}^{\delta}\right]\to 0$ as $t\to\infty$. The sequence $\{V_{t}^{\delta}\}$ is non-negative, so by Fatou's Lemma, we have,
	\begin{align}
	\label{eq:l1_14}
	0\leq \mathbb{E}_{\btheta}\left[V^{\delta}\right] \le \liminf_{t\to\infty}\mathbb{E}_{\btheta}\left[V_{t}^{\delta}\right] = 0.
	\end{align}
	\noindent Hence $V^{\delta}=0$ a.s. and thus $\left\|\mathbf{z}_{t}^{\delta}\right\|\to 0$ as $t\to\infty$ and the assertion follows.		
\end{proof}
We will use the following approximation result~(Lemma \ref{l33}) and the generalized convergence criterion~(Lemma \ref{l34}) for the proof of Theorem \ref{th:cons}.
\noindent\begin{Lemma}[Lemma 4.3 in \citet{Fabian-1}]
	\label{l33}
	Let $\{b_{t}\}$ be a scalar sequence satisfying
\begin{align}
		\label{eq:l33_1}
		b_{t+1}\le \left(1-\frac{c}{t+1}\right)b_{t}+d_{t}(t+1)^{-\tau},
		\end{align}
	where $c > \tau, \tau > 0$, and the sequence $d_{t}$ is summable. Then, we have,
\begin{align}
		\label{eq:l33_2}
		\limsup_{t\to\infty}~(t+1)^{\tau}b_{t}<\infty.
		\end{align}
\end{Lemma}
\noindent\begin{Lemma}[Lemma 10 in \citet{dubins1965sharper}]
	\label{l34}
	Let $\{J(t)\}$ be an $\mathbb{R}$-valued $\{\mathcal{F}_{t+1}\}$-adapted process such that $\mathbb{E}\left[J(t)|\mathcal{F}_{t}\right]]=0$ a.s. for each $t\geq 1$. Then the sum $\sum_{t\geq 0}J(t)$ exists and is finite a.s. on the set where $\sum_{t\geq 0}\mathbb{E}\left[J(t)^{2}|\mathcal{F}_{t}\right]$ is finite.
\end{Lemma}
Lemma \ref{le:l1} establishes the almost sure convergence of the averaged estimate sequence $\{\mathbf{x}_{\mbox{avg}}(t)\}$ to the true underlying parameter. We now establish the order optimal convergence of the estimate sequence in terms of $t$.
\begin{proof}[Proof of Theorem \ref{th:cons}]
	\noindent We first analyze the rate of convergence of the process $\{\mathbf{z}_{t}^{\delta}\}$ as developed in Lemma \ref{le:l1} and note that the rate of convergence of the process $\{\mathbf{z}_{t}^{\delta}\}$ suffices for the rate of convergence of the process $\{\mathbf{z}_{t}\}$. For each $\delta > 0$, recall the process $\{\mathbf{z}_{t}^{\delta}\}$ as in \eqref{eq:l1_2}-\eqref{eq:l1_7}. Let $\overline{\tau}\in[0,1/2)$ be such that,
	\begin{align}
	\label{eq:thc_1}
	\mathbb{P}_{\btheta}\left(\lim_{t\to\infty}(t+1)^{\overline{\tau}}\left\|\mathbf{z}_{t}^{\delta}\right\|=0\right)=1.
	\end{align}
	It is to be noted that such a $\overline{\tau}$ always exists from Lemma \ref{le:l1}. We now focus on showing that there exists $\tau$ such that $\overline{\tau}<\tau <1/2$ for which the assertion holds. Define $\widetilde{\tau}\in (\tau,1/2)$ and $\mu =\frac{1}{2}(\overline{\tau}+\widetilde{\tau})$. Then, for each $\delta > 0$,
	\begin{align}
	\label{eq:thc_2}
	&\left\|\mathbf{z}_{t+1}^{\delta}\right\|^{2} \le \left\|\mathbf{I}_{M}-\alpha_{t}\mathbf{\Gamma}\right\|^{2}\left\|\mathbf{z}_{t}^{\delta}\right\|^{2}+\alpha_{t}^{2}\left\|\mathbf{U}^{\delta}_{t}\right\|^{2}\nonumber\\&+\alpha_{t}^{2}\left\|\mathbf{J}_{t}\right\|^{2}\nonumber\\
	&+2\alpha_{t}\left(\mathbf{z}_{t}^{\delta}\right)^{\top}\left(\mathbf{I}_{M}-\alpha_{t}\mathbf{\Gamma}\right)\mathbf{J}_{t}\nonumber\\&+2\alpha_{t}\left\|\mathbf{U}^{\delta}_{t}\right\|\left(\left\|\mathbf{I}_{M}-\alpha_{t}\mathbf{\Gamma}\right\|\left\|\mathbf{z}_{t}^{\delta}\right\|+\alpha_{t}\left\|\mathbf{J}_{t}\right\|\right).
	\end{align}
	We have that, $1 > \tau_{1}+\frac{1}{2+\epsilon_{1}}+\frac{1}{2}$, hence the process $\{\mathbf{U}_{t}^{\delta}\}$ may be chosen such that, $\left\|\mathbf{U}_{t}^{\delta}\right\|= o\left((t+1)^{-1/2}\right)$. Moreover, as $\left\|\mathbf{z}_{t}^{\delta}\right\|= o\left((t+1)^{-\overline{\tau}}\right)$, we have,
	\begin{align}
	\label{eq:thc_3}
	2\alpha_{t}\left\|\mathbf{U}^{\delta}_{t}\right\|\left\|\mathbf{I}_{M}-\alpha_{t}\mathbf{\Gamma}\right\|\left\|\mathbf{z}_{t}^{\delta}\right\| = o\left((t+1)^{-3/2-\overline{\tau}}\right).
	\end{align}
	From Assumption \ref{m:1}, we have that,
	\begin{align}
	\label{eq:thc_4}
	\mathbb{P}_{\btheta}\left(\lim_{t\to\infty}(t+1)^{-1/2-\epsilon}\left\|\mathbf{J}_{t}^{\delta}\right\|\right)=1,~\textit{for~each}~\epsilon > 0,
	\end{align}
	and hence we conclude that
	\begin{align}
	\label{eq:thc_5}
	2\alpha_{t}^{2}\left\|\mathbf{U}^{\delta}_{t}\right\|\left\|\mathbf{J}_{t}\right\| = o\left((t+1)^{-3/2-\overline{\tau}}\right).
	\end{align}
	Since, $2\mu=\overline{\tau}+\widetilde{\tau}$ and $\widetilde{\tau}<1/2$, we have the following conclusions
	\begin{align}
	\label{eq:thc_6}
	&\sum_{t\geq 0}(t+1)^{2\mu}\alpha_{t}\left\|\mathbf{U}^{\delta}_{t}\right\|\left\|\mathbf{I}_{M}-\alpha_{t}\mathbf{\Gamma}\right\|\left\|\mathbf{z}_{t}^{\delta}\right\| <\infty \nonumber\\
	&\sum_{t\geq 0}(t+1)^{2\mu}\alpha_{t}^{2}\left\|\mathbf{U}^{\delta}_{t}\right\|\left\|\mathbf{J}_{t}\right\|<\infty \nonumber\\
	&\sum_{t\geq 0}(t+1)^{2\mu}\alpha_{t}^{2}\left\|\mathbf{U}^{\delta}_{t}\right\|^{2} < \infty\nonumber\\
	&\sum_{t\geq 0}(t+1)^{2\mu}\alpha_{t}^{2}\left\|\mathbf{J}_{t}\right\|^{2} < \infty.
	\end{align}
	With the above development in place, let $\{W_{t}^{\delta}\}$ denote the $\mathcal{F}_{t+1}$-adapted sequence given by
	\begin{align}
	\label{eq:thc_7}
	W_{t}^{\delta}=\alpha_{t}\left(\mathbf{z}_{t}^{\delta}\right)^{\top}\left(\mathbf{I}-\alpha_{t}\mathbf{\Gamma}\right)\mathbf{J}_{t},
	\end{align}
	where $\mathbb{E}_{\btheta}\left.\left[W_{t}^{\delta}\right|\mathcal{F}_{t}\right]=0$ and for $t$ chosen sufficiently large, we have that,
	\begin{align}
	\label{eq:thc_8}
	&\mathbb{E}_{\btheta}\left.\left[\left(W_{t}^{\delta}\right)^{2}\right|\mathcal{F}_{t}\right] = o\left((t+1)^{-2-2\overline{\tau}}\right)\nonumber\\
	&\Rightarrow \mathbb{E}_{\btheta}\left.\left[(t+1)^{4\mu}\left(W_{t}^{\delta}\right)^{2}\right|\mathcal{F}_{t}\right] =o\left((t+1)^{-2-2\overline{\tau}+4\mu}\right)\nonumber\\&=o\left((t+1)^{-2+2\widetilde{\tau}}\right).
	\end{align}
	Since, $2\widetilde{\tau}<1$, the sequence $\mathbb{E}_{\btheta}\left.\left[(t+1)^{4\mu}\left(W_{t}^{\delta}\right)^{2}\right|\mathcal{F}_{t}\right]$ is summable and by Lemma \ref{l34}, $\sum_{t\geq0}(t+1)^{2\mu}W_{t}^{\delta}$ exists. It may be shown that as $\alpha_{t}\to 0$ as $t\to\infty$,
	\begin{align}
	\label{eq:thc_9}
	\left\|\mathbf{I}-\alpha_{t}\mathbf{\Gamma}\right\|^{2} \le 1-c_{1}\alpha_{t},
	\end{align}
	where $c_{1}=\lambda_{min}\left(\mathbf{\Gamma}\right)$.
	Then, from \eqref{eq:thc_2}, we have,
	\begin{align}
	\label{eq:thc_10}
	\left\|\mathbf{z}_{t+1}^{\delta}\right\|^{2} \le \left(1-c_{1}\alpha_{t}\right)\left\|\mathbf{z}_{t}^{\delta}\right\|^{2}+d_{t}(t+1)^{-2\mu},
	\end{align}
	where the term $d_{t}(t+1)^{-2\mu}$ represents all the residual terms in \eqref{eq:thc_2}. The fact that $\lim_{t\to\infty}\sum_{s=0}^{t}d_{s}$ exists and is finite in conjunction with $c_{1}\alpha_{t}(t+1)\geq 1\geq 2\mu$~(from Assumption \ref{m:4}) brings \eqref{eq:thc_10} under the purview of Lemma \ref{l33} and yields
	\begin{align}
	\label{eq:thc_11}
	\limsup_{t\to\infty}(t+1)^{2\mu}\left\|\mathbf{z}_{t}^{\delta}\right\|^{2} < \infty~a.s.,
	\end{align}
	which leads to the conclusion that there exists $\tau$ with $\overline{\tau}<\tau<\mu$, such that $(t+1)^{\tau}\left\|\mathbf{z}_{t}^{\delta}\right\|\to 0$ as $t\to\infty$. The fact that the above development holds for all $\delta > 0$, we conclude that $(t+1)^{\tau}\left\|\mathbf{z}_{t}\right\|\to 0$ as $t\to\infty$. Hence, for every $\overline{\tau}$ for which
	\begin{align}
	\label{eq:thc_12}
	\mathbb{P}_{\btheta}\left(\lim_{t\rightarrow\infty}(t+1)^{\overline{\tau}}\|\mathbf{x}_{\mbox{avg}}(t)-\btheta\|=0\right)=1
	\end{align}
	holds, then there exists $\tau \in \left(\overline{\tau},1/2\right)$ for which the convergence continues to hold. Finally, an application of induction yields the result
	\begin{align}
	\label{eq:thc_13}
	\mathbb{P}_{\btheta}\left(\lim_{t\rightarrow\infty}(t+1)^{\tau}\|\mathbf{x}_{\mbox{avg}}(t)-\btheta\|=0\right)=1, \forall \tau\in[0,1/2)
	\end{align}
	The above result in conjunction with Lemma \ref{le:l2} and the usage of triangle inequality yields $\forall \tau\in[0,1/2)$
	\begin{align}
	\label{eq:thc_14}
	&(t+1)^{\tau}\left\|\mathbf{x}_{n}(t)-\btheta\right\|\le(t+1)^{\tau}\left\|\mathbf{x}_{\mbox{avg}}(t)-\btheta\right\|\nonumber\\&+(t+1)^{\tau}\left\|\mathbf{x}_{n}(t)-\mathbf{x}_{\mbox{avg}}(t)\right\|\nonumber\\
	&\Rightarrow\lim_{t\to\infty}(t+1)^{\tau}\left\|\mathbf{x}_{n}(t)-\btheta\right\| = 0~a.s.
	\end{align}
\end{proof}
\begin{proof}[Proof of Theorem \ref{th:credo}]
	Proceeding as in proof of Lemma \ref{le:conv}, we have, for $t$ large enough
	\begin{align}
	\label{eq:lc_pr_13}
	&\mathbb{E}_{\btheta}[V(t+1)|\mathcal{F}_{t}]\le\left(1-2c_{4}\alpha_{t}+c_{7}\alpha^{2}_{t}\right)V(t)+c_{6}\alpha_{t}^{2}\nonumber\\
	&\le V(t)+c_{6}\alpha_{t}^{2},
	\end{align}
	\noindent as for $t$ large enough, $-c_{4}\alpha_{t}+c_{7}\alpha^{2}_{t}<0$. Before proceeding further, we note that, from \eqref{eq:ineq1},
	\begin{align}
	\label{eq:det1_5}
	&\mathbf{x}^{\top}\left(\beta_{t}\left(\overline{\mathbf{L}}\otimes\mathbf{I}_{M}\right)+\alpha_{t}\mathbf{G}_{H}\mathbf{\Sigma}^{-1}\mathbf{G}_{H}^{\top}\right)\mathbf{x}\nonumber\\
	&=\alpha_{t}\mathbf{x}^{\top}\left(\frac{\beta_{t}}{\alpha_{t}}\left(\overline{\mathbf{L}}\otimes\mathbf{I}_{M}\right)+\mathbf{G}_{H}\mathbf{\Sigma}^{-1}\mathbf{G}_{H}^{\top}\right)\mathbf{x}\nonumber\\
	&\geq\alpha_{t}\mathbf{x}^{\top}\left(\left(\mathbf{L}\otimes\mathbf{I}_{M}\right)+\mathbf{G}_{H}\mathbf{\Sigma}^{-1}\mathbf{G}_{H}^{\top}\right)\mathbf{x}\geq c_{4}\alpha_{t},
	\end{align}
	where
	\begin{align}
	\label{eq:det1_55}
	c_{4}=\lambda_{\mbox{\scriptsize{min}}}\left(\left(\overline{\mathbf{L}}\otimes\mathbf{I}_{M}\right)+\mathbf{G}_{H}\mathbf{\Sigma}^{-1}\mathbf{G}_{H}^{\top}\right).
	\end{align}
	Thus, we have that
	\begin{align}
	\label{eq:det1_8_1}
	\left\|\mathbf{I}_{NM}-\beta_{t}\left(\overline{\mathbf{L}}\otimes\mathbf{I}_{M}\right)-\alpha_{t}\mathbf{G}_{H}\mathbf{\Sigma}^{-1}\mathbf{G}_{H}^{\top}\right\|\le 1-c_{4}\alpha_{t},
	\end{align}
	for all $t\geq t_{1}$, where $t_{1}$ is chosen to be appropriately large.
	\noindent Now, consider the $\{\mathcal{F}_{t}\}$-adapted process $\{V_{1}(t)\}$ defined as follows
	\begin{align}
	\label{eq:lc_pr_14}
	&V_{1}(t)=V(t)+c_{6}\sum_{s=t}^{\infty}\alpha_{s}^{2}\nonumber\\
	&=V(t)+c_{8}\sum_{s=t}^{\infty}(t+1)^{-2},
	\end{align}		
	\noindent for appropriately chosen positive constant $c_{8}$.Since, $\{(t+1)^{-2}\}$ is summable, the process $\{V_{1}(t)\}$ is bounded from above. Moreover, it also follows that $\{V_{1}(t)\}_{t\geq t_{1}}$ is a supermartingale and hence converges a.s. to a finite random variable. By definition from \eqref{eq:l1_pr_14}, we also have that $\{V(t)\}$ converges to a non-negative finite random variable $V^{*}$. Finally, from \eqref{eq:lc_pr_13}, we have that,
	\begin{align}
	\label{eq:lc_pr_15}
	&\mathbb{E}_{\btheta}[V(t+1)]\le \left(1-c_{4}\alpha_{t}\right)\mathbb{E}_{\btheta}[V(t)]+c_{8}(t+1)^{-2}\nonumber\\
	&\Rightarrow {E}_{\btheta}[V(t+1)]\le \left(1-c_{4}\alpha_{t}\right)\mathbb{E}_{\btheta}[V(t)]+c_{10}\alpha_t(t+1)^{-1}
	\end{align}
	\noindent for $t\geq t_{1}$. The summability of $\{\alpha_{t}\}$ in conjunction with assumption \ref{m:4} ensures that the sequence $\{V(t)\}$ then falls under the purview of Lemma \ref{l33}, and we have
	\begin{align}
	\label{eq:lc_pr_16}
	&\limsup_{t\to\infty}(t+1)\mathbb{E}_{\btheta}[V(t+1)] < \infty \nonumber\\
	&\Rightarrow \mathbb{E}_{\btheta}[V(t)] = O\left(\frac{1}{t}\right).
	\end{align}
	Furthermore, from \eqref{eq:lc_pr_14}, we also have that
	\begin{align}
	\label{eq:lc_pr_17}
	&\mathbb{E}_{\btheta}[V_{1}(t)]\le \mathbb{E}_{\btheta}[V(t)]+\frac{c_{6}\pi^{2}}{6}\nonumber\\
	&\Rightarrow \mathbb{E}_{\btheta}[\left\|\mathbf{x}_{n}(t)-\btheta\right\|^{2}] = O\left(\frac{1}{t}\right).
	\end{align}
	\noindent It is to be noted that the communication cost $\mathcal{C}_{t}$ for the proposed $\mathcal{CREDO}$ algorithm, is given by $\mathcal{C}_{t} = \Theta\left(t^{1+\frac{\epsilon-\tau_{1}}{2}}\right)$ and thus the assertion follows in conjunction with \eqref{eq:lc_pr_17}.	
\end{proof}
\subsection{Asymptotic Normality and Covariance}
The proof of Theorem \ref{th:2} needs the following Lemma from \citet{Fabian-2} concerning the asymptotic normality of the stochastic recursions.
\begin{Lemma}[Theorem 2.2 in \citet{Fabian-2}]
	\label{main_res_l0}
	\noindent Let $\{\mathbf{z}_{t}\}$ be an $\mathbb{R}^{k}$-valued $\{\mathcal{F}_{t}\}$-adapted process that satisfies
	\begin{align}
	\label{eq:l0_1}
	&\mathbf{z}_{t+1}=\left(\mathbf{I}_{k}-\frac{1}{t+1}\Gamma_{t}\right)\mathbf{z}_{t}+(t+1)^{-1}\mathbf{\Phi}_{t}\mathbf{V}_{t}\nonumber\\&+(t+1)^{-3/2}\mathbf{T}_{t},
	\end{align}
	\noindent where the stochastic processes $\{\mathbf{V}_{t}\}, \{\mathbf{T}_{t}\} \in \mathbb{R}^{k}$ while $\{\mathbf{\Gamma}_{t}\}, \{\mathbf{\Phi}_{t}\} \in \mathbb{R}^{k\times k}$. Moreover, suppose for each $t$, $\mathbf{V}_{t-1}$ and $\mathbf{T}_{t}$ are $\mathcal{F}_{t}$-adapted, whereas the processes $\{\mathbf{\Gamma}_{t}\}$, $\{\mathbf{\Phi}_{t}\}$ are $\{\mathcal{F}_{t}\}$-adapted.	
	\noindent Also, assume that
	\begin{align}
	\label{eq:l0_2}
	\mathbf{\Gamma}_{t}\to\mathbf{\Gamma}, \mathbf{\Phi}_{t}\to\mathbf{\Phi},~ \textit{and} ~\mathbf{T}_{t}\to 0 ~~\mbox{a.s. as $t\rightarrow\infty$},
	\end{align}
	\noindent where $\mathbf{\Gamma}$ is a symmetric and positive definite matrix, and admits an eigen decomposition of the form $\mathbf{P}^{\top}\mathbf{\Gamma}\mathbf{P}=\mathbf{\Lambda}$, where $\mathbf{\Lambda}$ is a diagonal matrix and $\mathbf{P}$ is an orthogonal matrix. Furthermore, let the sequence $\{\mathbf{V}_{t}\}$ satisfy $\mathbb{E}\left[\mathbf{V}_{t}|\mathcal{F}_{t}\right]=0$ for each $t$ and suppose there exists a positive constant $C$ and a matrix $\Sigma$ such that $C > \left\|\mathbb{E}\left[\mathbf{V}_{t}\mathbf{V}_{t}^{\top}|\mathcal{F}_{t}\right]-\Sigma\right\|\to 0~a.s. ~\textit{as}~ t\to\infty$ and with $\sigma_{t,r}^{2}=\int_{\left\|\mathbf{V}_{t}\right\|^{2} \ge r(t+1)}\left\|\mathbf{V}_{t}\right\|^{2}d\mathbb{P}$, let $\lim_{t\to\infty}\frac{1}{t+1}\sum_{s=0}^{t}\sigma_{s,r}^{2}=0$ for every $r > 0$. Then, we have,
	\begin{align}
	\label{eq:l0_3}
	(t+1)^{1/2}\mathbf{z}_{t}\overset{\mathcal{D}}{\Longrightarrow}\mathcal{N}\left(\mathbf{0}, \mathbf{P}\mathbf{M}\mathbf{P}^{\top}\right),
	\end{align}
	\noindent where the $(i,j)$-th entry of the matrix $\mathbf{M}$ is given by
	\begin{align}
	\label{eq:l0_4}
	\left[\mathbf{M}\right]_{ij}=\left[\mathbf{P}^{\top}\mathbf{\Phi}\mathbf{\Sigma}\mathbf{\Phi}^{\top}\mathbf{P}\right]_{ij}\left(\left[\mathbf{\Lambda}\right]_{ii}+\left[\mathbf{\Lambda}\right]_{jj}-1\right)^{-1}.
	\end{align}
\end{Lemma}
In order to establish asymptotic normality and characterize the estimator in terms of asymptotic covariance, the following Lemma plays a crucial role.
\begin{proof}[Proof of Theorem \ref{th:2}]
	We invoke the definition of the process $\{\mathbf{z}_{t}\}$ as defined in \eqref{eq:l1_2}-\eqref{eq:l1_3}. We rewrite the recursion for $\{\mathbf{z}_{t}\}$ as follows:
	\begin{align}
	\label{eq:l3_2}
	\mathbf{z}_{t+1} = \left(\mathbf{I}_{M}-\alpha_{t}\Gamma_{t}\right)\mathbf{z}_{t}+(t+1)^{-3/2}\mathbf{T}_{t}+(t+1)^{-1}\mathbf{\Phi}_{t}\mathbf{V}_{t},
	\end{align}
	where
	\begin{align}
	\label{eq:l3_3}
	&\mathbf{\Gamma}_{t} = \mathbf{\Gamma} = \frac{1}{N}\sum_{n=1}^{N}\mathbf{H}_{n}^{\top}\mathbf{\Sigma}_{n}^{-1}\mathbf{H}_{n}\nonumber\\
	&\mathbf{T}_{t}=a(t+1)^{1/2}\mathbf{U}_{t}\nonumber\\& = \frac{a}{N}\sum_{n=1}^{N}\mathbf{H}_{n}^{\top}\mathbf{\Sigma}_{n}^{-1}(t+1)^{1/2}\left(\mathbf{x}_{n}(t)-\mathbf{x}_{\mbox{avg}}(t)\right)\to 0,~t\to\infty\nonumber\\
	&\mathbf{\Phi}_{t} = a\mathbf{I}\nonumber\\
	&\mathbf{V}_{t}=\mathbf{J}_{t}=\frac{1}{N}\sum_{n=1}^{N}\mathbf{H}_{n}^{\top}\mathbf{\Sigma}_{n}^{-1}\mathbf{\gamma}_{n}(t),~\mathbb{E}\left[\mathbf{V}_{t}|\mathcal{F}_{t}\right]=0,\nonumber\\&\mathbb{E}\left[\mathbf{V}_{t}\mathbf{V}^{\top}_{t}|\mathcal{F}_{t}\right]=\frac{1}{N^{2}}\sum_{n=1}^{N}\mathbf{H}_{n}^{\top}\mathbf{\Sigma}_{n}^{-1}\mathbf{H}_{n},
	\end{align}
	and the convergence of $\mathbf{T}_{t}$ follows from Lemma \ref{le:l2}. Due to the i.i.d nature of the noise process, we have the uniform integrability condition for the process $\{\mathbf{V}_{t}\}$. Hence, $\{\mathbf{x}_{\mbox{\scriptsize{avg}}}(t)\}$ falls under the purview of Lemma \ref{main_res_l0} and we thus conclude that
	\begin{align}
	\label{eq:l3_4}
	(t+1)^{1/2}\left(\mathbf{x}_{\mbox{\scriptsize{avg}}}(t)-\btheta\right)\overset{\mathcal{D}}{\Longrightarrow}\mathcal{N}(0,\mathbf{P}\mathbf{M}\mathbf{P}^{\top}),
	\end{align}
	\noindent where
	\begin{align}
	\label{eq:l3_5}
	&a\mathbf{P}^{\top}\mathbf{\Gamma}\mathbf{P}=a\mathbf{\Lambda},\nonumber\\
	&\left[\mathbf{M}\right]_{ij}=\left[a^{2}\mathbf{P}^{\top}\mathbf{\Phi}\left(\frac{1}{N^{2}}\sum_{n=1}^{N}\mathbf{H}_{n}^{\top}\mathbf{\Sigma}_{n}^{-1}\mathbf{H}_{n}\right)\mathbf{\Phi}^{\top}\mathbf{P}\right]_{ij}\nonumber\\&\times\left(a\left[\mathbf{\Lambda}\right]_{ii}+a\left[\mathbf{\Lambda}\right]_{jj}-1\right)^{-1}\nonumber\\&=\frac{a^{2}}{N}\left[\mathbf{\Lambda}\right]_{ij}\left(a\left[\mathbf{\Lambda}\right]_{ii}+a\left[\mathbf{\Lambda}\right]_{jj}-1\right)^{-1},
	\end{align}
	\noindent which also implies that $\mathbf{M}$ is a diagonal matrix with its $i$-th diagonal element given by $\frac{a^{2}\mathbf{\Lambda}_{ii}}{2aN\mathbf{\Lambda}_{ii}-N}$. Note that, Assumption \ref{m:4} ensures that $\frac{a^{2}\mathbf{\Lambda}_{ii}}{2aN\mathbf{\Lambda}_{ii}-N} > 0$,~$\forall i$. We already have that $\mathbf{P}\mathbf{\Lambda}\mathbf{P}^{\top}=\mathbf{\Gamma}$. Hence, the matrix with eigenvalues as {\small$\frac{a^{2}\mathbf{\Lambda}_{ii}}{2aN\mathbf{\Lambda}_{ii}-N}$} is given by
	\begin{align}
	\label{eq:l3_6}
	\mathbf{P}\mathbf{M}\mathbf{P}^{\top}=\frac{a\mathbf{I}}{2N}+\frac{\left(\mathbf{\Gamma}-\frac{\mathbf{I}}{2a}\right)^{-1}}{4N}.
	\end{align}
	\noindent Now from Lemma \ref{le:l2}, we have that the processes $\{\mathbf{x}_{n}(t)\}$ and $\{\mathbf{x}_{\mbox{\scriptsize{avg}}}(t)\}$ are indistinguishable in the $(t+1)^{1/2}$ time scale, which is formalized as follows:
	\begin{align}
		\label{eq:l3_7}
		&\mathbb{P}_{\btheta}\left(\lim_{t\to\infty}\left\|\sqrt{t+1}\left(\mathbf{x}_{n}(t)-\btheta\right)-\sqrt{t+1}\left(\mathbf{x}_{\mbox{\scriptsize{avg}}}(t)-\btheta\right)\right\|=0\right)\nonumber\\
		&=\mathbb{P}_{\btheta}\left(\lim_{t\to\infty}\left\|\sqrt{t+1}\left(\mathbf{x}_{n}(t)-\mathbf{x}_{\mbox{\scriptsize{avg}}}(t)\right)\right\|=0\right)=1.
		\end{align}
	\noindent Thus, the difference of the sequences $\left\{\sqrt{t+1}\left(\mathbf{x}_{n}(t)-\btheta\right)\right\}$ and $\left\{\sqrt{t+1}\left(\mathbf{x}_{\mbox{\scriptsize{avg}}}(t)-\btheta\right)\right\}$ converges a.s. to zero as $t\rightarrow\infty$ \noindent and hence we have,
	\begin{align}
	\label{eq:t2_pr_12}
	\sqrt{t+1}\left(\mathbf{x}_{n}(t)-\btheta\right)\overset{\mathcal{D}}{\Longrightarrow}\mathcal{N}\left(0,\frac{a\mathbf{I}}{2N}+\frac{\left(\mathbf{\Gamma}-\frac{\mathbf{I}}{2a}\right)^{-1}}{4N}\right).
	\end{align}
\end{proof}
\vskip 0.2in
\bibliographystyle{plainnat}
\bibliography{CentralBib,glrt,dsprt}
\end{document}